\documentclass[a4paper,reqno]{amsart}

\usepackage[utf8]{inputenc}
\usepackage[english]{babel}
\usepackage{amsmath,amssymb,amsthm,exscale,amsopn,mathtools,color,tikz-cd,cite,fullpage}
\usepackage[colorlinks=true,pdftex,unicode=true,linktocpage,bookmarksopen,hypertexnames=false]{hyperref}

\newtheorem{lem}{Lemma}[section]
\newtheorem{prop}[lem]{Proposition}
\newtheorem{thm}[lem]{Theorem}
\newtheorem{cor}[lem]{Corollary}

\newtheorem{theorem}{Theorem}
\newtheorem*{theorem*}{Theorem}

\theoremstyle{definition}
\newtheorem{ex}[lem]{Example}
\newtheorem{rem}[lem]{Remark}

\newcommand{\free}[1]{\langle #1 \rangle}

\newcommand{\mb}{\mathbb}
\newcommand{\mc}{\mathcal}
\newcommand{\mf}{\mathfrak}
\newcommand{\GAP}{\textsf{GAP}}

\newcommand{\s}{\subseteq}
\newcommand{\ov}{\overline}

\renewcommand{\le}{\leqslant}
\renewcommand{\ge}{\geqslant}

\DeclareMathOperator{\Alt}{Alt}
\DeclareMathOperator{\Aut}{Aut}
\DeclareMathOperator{\End}{End}
\DeclareMathOperator{\ga}{v}
\DeclareMathOperator{\GL}{GL}
\DeclareMathOperator{\gr}{gr}
\DeclareMathOperator{\Hol}{Hol}
\DeclareMathOperator{\id}{id}
\DeclareMathOperator{\Img}{Im}

\DeclareMathOperator{\Ker}{Ker}
\DeclareMathOperator{\lcm}{lcm}

\DeclareMathOperator{\op}{op}

\DeclareMathOperator{\Soc}{Soc}

\DeclareMathOperator{\Sym}{Sym}

\author[I. Colazzo \and E. Jespers \and {\L}. Kubat \and A. Van Antwerpen]
{Ilaria  Colazzo \and Eric Jespers \and {\L}ukasz Kubat \and Arne Van Antwerpen}

\address[I. Colazzo (ORCID: 0000-0002-2713-0409)]{School of Mathematics, University of Leeds, Leeds LS2 9JT, UK \&
Department of Mathematics and Data Science, Vrije Universiteit Brussel, Pleinlaan 2, 1050 Brussel}
\email{I.Colazzo@leeds.ac.uk}

\address[E. Jespers (ORCID: 0000-0002-2695-7949)]
{Department of Mathematics and Data Science, Vrije Universiteit Brussel, Pleinlaan 2, 1050 Brussel}
\email{Eric.Jespers@vub.be}

\address[{\L}. Kubat (ORCID: 0000-0002-7848-6405)]{University of Warsaw, Institute of Mathematics, Banacha 2, 02-097 Warsaw, Poland}
\email{Lukasz.Kubat@mimuw.edu.pl}

\address[A. Van Antwerpen (ORCID: 0000-0001-7619-6298)]
{Department of Mathematics and Statistics, National University of Ireland - Maynooth, Maynooth, Ireland}
\email{Arne.VanAntwerpen@mu.ie}

\title{Simple solutions of the Yang--Baxter equation}
\date{}

\begin{document}

\begin{abstract}
    We study  simple set-theoretic solutions of the Yang--Baxter equation that are finite and non-degenerate. Such retractable
    solutions are fully described and to investigate the irretracble solutions we give a new algebraic method. Our approach includes
    and extends the work of Joyce for quandles and Castelli for involutive solutions, demonstrating that the simplicity of a solution
    can be understood through its associated permutation skew left brace. In particular, we show that this skew left brace must have
    the smallest non-zero ideal, and the quotient by this ideal gives a trivial skew left brace of cyclic type; clearly all simple skew
    left braces satisfy these assumptions. As an application of our approach we construct and characterise  new infinite families of
    simple solutions that are neither involutive nor quandles. Additionally, we show that our method can be applied to simple skew left
    braces to generate further families of simple solutions.
\end{abstract}

\maketitle

\section{Introduction}

The Yang--Baxter equation has its origins in the works of Yang \cite{Ya} and Baxter \cite{Ba}. It has significant implications not
only in mathematical physics but also in various branches of pure mathematics, including quantum groups, Hopf algebras, and knot theory
\cite{BG,KasselBook,Manin,MR0901157,MR0638121}. A comprehensive description of all solutions of the Yang--Baxter equation remains
currently beyond reach. An interesting family of solutions consists of those where the solution acts on a basis. Drinfeld referred
to such solutions as set-theoretic in \cite{Dri1992}. A \emph{set-theoretic solution of the Yang--Baxter equation} is a pair $(X,r)$,
where $X$ is a non-empty set and $r\colon X\times X\to X\times X$ is a map satisfying the \emph{braid equation}
$(r\times{\id})({\id}\times r)(r\times{\id})=({\id}\times r)(r\times {\id})({\id}\times r)$ on $X\times X\times X$. The class of
non-degenerate set-theoretic solutions holds particular importance. As usual we write \[r(x,y)=(\lambda_x(y),\rho_y(x))\] and thus
non-degeneracy means that each $\lambda_x$ and $\rho_y$ is a bijection. Recall that by \cite[Theorem 3.1]{CoJeVAVe21x} non-degeneraracy
implies that $r$ is necessarily a bijective map provided $X$ is finite. Recently, Jedli\v{c}ka and Pilitowska
\cite[Corollary 3.4]{diagonals} have shown that any non-degenerate solution is bijective. Hence, for simplicity, we refer
to non-degenerate bijective solutions simply as non-degenerate solutions. Pioneer works \cite{GVdB,ESS99} brought
algebraic tools to study a subclass of these solutions, the non-degenerate set-theoretic solutions that are \emph{involutive}, i.e., also
$r^2=\id$. Later in \cite{LYZ00,MR1809284} these methods were extended to the class of all non-degenerate solutions. Given a non-degenerate
set-theoretic solution of the Yang--Baxter $(X,r)$, the \emph{(left) derived solution} of $(X,r)$, for which we will use the suggestive
notation $(X,r')$, is given by \[r'(x,y)=(y,\sigma_y(x)),\quad\text{where}\quad\sigma_y(x)=\lambda_y\rho_{\lambda_x^{-1}(y)}(x).\]
It is well-known that $(X,r')$ also is a non-degenerate set-theoretic solution, see for instance
\cite[Proposition 5.7]{MR3558231}. It turns out that the derived solution plays a crucial role in the study of
non-degenerate solutions, for example, being involutive or being injective for a non-degenerate
solution can be read from its derived solution \cite{MR1809284,MR3974961}. It is easy to verify
that $X$ with respect to the binary operation $\triangleleft$ defined by $x\triangleleft y=\sigma_y(x)$ is a rack,
i.e., $\sigma_y$ is bijective for all $y\in X$ and for any $x,y,z\in X$ it holds
$(x\triangleleft y)\triangleleft z=(x\triangleleft z)\triangleleft (y\triangleleft z)$. A rack $(X,\triangleleft)$
is said to be a \emph{quandle} if $x \triangleleft x=x$ for each $x\in X$. Any group $B$ with the operation given
by $g\triangleleft h=h^{-1}gh$ defines a rack operation on $B$, which is clearly a quandle operation. Every subrack
of such a quandle is called \emph{conjugation quandle}. On the other hand, a rack $(X,\triangleleft)$ defines
a set-theoretic solution of the Yang--Baxter equation by $r(x,y)=(y,x\triangleleft y)$. This relation between
solutions of the Yang--Baxter equation and racks has been noticed by Brieskorn in \cite{MR0975077}. The notion
of racks has been considered in the literature, mainly as a way to produce knot invariants
\cite{MR0638121,MR1778150,MR0672410,MR0975077} and in the study of Nichols algebras over a group \cite{MR1994219}.

A characterisation of arbitrary finite non-degenerate solutions is presently beyond reach. Hence one needs to restrict
the investigations to building blocks of such solutions. Of course a first and obvious restriction is to deal with
indecomposable solutions. Recall that a solution $(X,r)$  is said to be \emph{decomposable} if $X=Y\cup Z$, a disjoint
union of non-empty subsets, and $r$ restricted to $Y\times Y$ and $Z\times Z$ respectively defines a solution on $Y$ and $Z$. 

Another obvious restriction is to simple solutions, i.e., a solution $(X,r)$ of which any epimorphic image is either isomorphic
to $(X,r)$ or to a solution $(Y,s)$ with $Y$ a singleton. Clearly any finite non-degenerate solution has an epimorphic image that is simple. 

We note that
a non-degenerate solution $(X,r)$ with $X$ of cardinality $2$ is simple. Moreover, if the cardinality of $X$ is at least $3$  and $(X,r)$ is
simple then the solution $(X,r)$ is also indecomposable. Simple solutions are the smallest non-trivial solutions and thus are the
building blocks of all solutions. Note that the only simple decomposable solution $(X,r)$  is the flip map on two elements,
i.e., $r(x,y)=(y,x)$, therefore throughout the paper we always assume that $|X|\ge 3$ and thus a simple solution is indecomposable.
The definition
of simple non-degenerate derived solutions appeared in \cite{MR0682881}, where a group-theoretic characterisation of simple
quandles has been obtained. The definition of simple solution seems to have appeared first in \cite{MR1994219}, where the
classification of simple racks has been improved given a concrete characterisation of which groups give rise to a simple rack.
Concerning involutive non-degenerate solutions, in \cite{MR3437282}, Vendramin introduced a notion of simple involutive
non-degenerate solutions (using the language of cycle sets). This definition of simplicity, however, coincides with the original
one only within the class of finite indecomposable non-degenerate involutive solutions. Ced\'o  and Okni\'nski \cite{CeOkn,MR4161288}
(see also \cite{MR4300920}) constructed finite simple non-degenerate involutive solutions. Recently in \cite{CO2022} they proved
for a finite simple non-degenerate involutive solution $(X,r)$ that if $|X|$ is not a prime power then $|X|$ cannot be square-free.
Moreover, if $|X|$ is a prime power $p^n$ then $n=1$, and there exist simple solutions of order $p$ and there is only one involutive
up to isomorphism. In \cite{MR0682881} Joyce described simplicity for quandles in a group theoretic setting  and Castelli,
in \cite{Cast22}, characterised the simplicity of non-degenerate involutive solutions in the algebraic setting of skew left braces. 

Recall that Rump \cite{Rump2007} and Guarnieri and Vendramin \cite{GV17} introduced skew left braces to provide an algebraic
framework for studying arbitrary non-degenerate set-theoretic solutions of the Yang--Baxter equation. A \emph{skew
left brace} is a triple $(B,+,\circ)$ such that both $(B,+)$ and $(B,\circ)$ are groups and \[a\circ(b+c)=a\circ b-a+a\circ c\] holds
for all $a,b,c\in B$. The theory of Hopf--Galois extensions suggests to use the following terminology based on properties of the additive
group. One says that a skew left brace $B$ is of \emph{abelian type} if $(B,+)$ is an abelian group (Rump in \cite{Rump2007}
simply names it a \emph{left brace}) and of \emph{cyclic type} if $(B,+)$ is a cyclic group. Since their introduction, skew left
braces have been widely studied \cite{MR3974961,MR4300920,MR4399130,MR3763907}, and recently are popularised in \cite{whatis}.
Every skew left brace $B$ naturally provides a set-theoretic solution $(B,r_B)$ of the Yang--Baxter equation that is non-degenerate
and skew left braces of abelian type provide involutive solutions. Specifically, $r_B\colon B\times B\to B\times  B$ is given by
\[r_B(a,b)=(-a+a\circ b, (-a+a\circ b)^{-1}\circ a\circ b),\] where $c^{-1}$ denotes the inverse of $c\in B$ in the group $(B,\circ)$.
Two fundamental skew left braces are associated with a finite non-degenerate solution $(X,r)$: the \emph{structure skew left brace}
$G(X,r)$ and the \emph{permutation skew left brace} $\mc{G}(X,r)$. In the next section we recall rigorous definitions of these
structures. Let us just mention that if the solution $(X,r)$ is finite so is the permutation skew left brace. Moreover,
if two non-degenerate solutions $(X,r)$ and $(Y,s)$ are isomorphic, then we have $G(X,r)\cong G(Y,s)$ as skew left braces,
as well as the permutation groups $\mc{G}(X,r)\cong \mc{G}(Y,s)$. But the converse is not true. Given a finite skew
left brace $B$, Bachiller, Cedó, and Jespers \cite{B2018,BCJ2016} provided a construction of all finite non-degenerate 
set-theoretic solutions $(X,r)$ such that $\mc{G}(X,r)$ and $ B$ are isomorphic as skew left braces. Providing more evidence that
skew left braces are the right algebraic tool to investigate all non-degenerate solutions of the Yang--Baxter equation.
Within this framework, there is the need to investigate its building blocks, the simple skew left braces, i.e., skew left braces
with no non-trivial homomorphic images. Note that in contrast with finite simple groups, it has been shown \cite{MR4161288}
that there is an abundance of finite simple skew left braces of abelian type \cite{MR4161288,MR4122077}. Moreover, more simple
skew left braces of abelian type have been constructed in \cite{MR3812099}. Only recently, Byott \cite{Byott-IP} constructed a
family of finite simple skew left braces that are not of abelian type. The first instances of this family can be found in the
\GAP\ package \emph{YangBaxter} \cite{YangBaxter}, i.e., \texttt{SmallSkewbrace(12,22)} and \texttt{SmallSkewbrace(12,23)}.

Okni\'nski and Ced\'o \cite{CeOkn} proved that some finite simple skew left braces of abelian type determine an involutive
simple solution $(X,r)$. They showed that finite simple braces which are additively generated by an orbit $X$ under the action
of the permutation skew left brace $\mc{G}(X,r)$ yield simple non-degenerate involutive solutions on $X$. This connection
between involutive simple solutions and simple skew left braces of abelian type recently has been refined by Castelli in
\cite{Cast22} by showing that the simplicity of a finite non-degenerate involutive solution $(X,r)$ is described by the
algebraic structure of its permutation skew left brace.

Problem 10, posed by Vendramin in \cite{MR3974481}, asks whether the simplicity of an involutive solution $(X,r)$ can be read
in terms of the left brace $G(X,r)$ and expresses interest in developing the theory of non-involutive simple solutions,
particularly in understanding how simplicity can be reflected in the skew left brace where the solution is realised.

In this paper, we address this problem by developing a new algebraic tool to describe finite simple non-degenerate solutions
to the Yang-–Baxter equation. Our approach provides a comprehensive solution, offering a framework that resolves the issue of
reflecting simplicity in skew left braces and extends existing results. While similar work has been done for specific subclasses
of solutions, such as quandles and involutive solutions, through the works of Joyce \cite{MR0638121} and Castelli \cite{Cast22},
our method not only recovers these approaches but also advances them. We also apply our method to construct new infinite families
of simple solutions.

The first step is to identify two quite distinct natural families of simple non-degenerate finite  solutions:
\begin{itemize}
    \item \textbf{Lyubashenko type solutions}, which are fully classified using combinatorial methods, see Section~\ref{secL}, and
    \item solutions that arise as \textbf{subsolutions of those associated with skew left braces} that have a clear algebraic structure,
    see part (2) in the following theorem.
\end{itemize}

This can be summarised as follows: we use the natural action of a skew left brace $B$ on itself, i.e., the action with respect
to the group generated by the maps $\lambda_a$ and $\sigma_a$ with $a\in B$. Such an action can be restricted to ideals and subsets
of $B$ that define subsolutions. The formulation also makes use of the ideal $B^{(2)}=B*B$, i.e., the additive subgroup of $B$
generated by all the elements $a*b=-a+a\circ b-b$ for $a,b\in B$, and the ideal $B^{(3)}=B^{(2)}*B$, that is the additive subgroup
of $B$ generated by all elements of the form $x*y$ with $x\in B^{(2)}$ and $y\in B$.

\begin{theorem}\label{CorSummary}
    A finite non-degenerate solution $(X,r)$ of the Yang--Baxter equation is simple if and only if it is one of the following types:
    \begin{enumerate}
     \item it is retractable and a Lyuabashenko solution,
     \item it is irretractable and it  is embedded into the associated solution $(B,r_B)$ of a finite skew left brace $B$ with the following properties:
        \begin{itemize}
            \item $B$ is additively generated by $X$,
            \item $B$ has the smallest non-zero ideal, say $V$ (and it turns out that $V$ is generated, as an ideal, by the elements $x-y$ with $x,y\in X)$,
            \item $B/V$ is a trivial skew left brace of cyclic type, 
            \item the action of the ideal $V$ on $X$ is transitive.
        \end{itemize}
            In particular, either $B^{(2)}=0$ and thus $V=[B,B]$, or $B^{(2)}\ne 0$ and $B^{(3)}=0$ and thus $V=B^{(2)}$, or $B^{(2)}=B^{(3)}\ne 0$.
    \end{enumerate}
\end{theorem}

Recall that a non-degenerate solution $(X,r)$ is said to be \emph{Lyubashenko} (see \cite{Dri1992}) if $r(x,y)=(\lambda(y),\rho(x))$
for all $x,y\in X$, where $\lambda$ and $\rho$ are commuting permutations on $X$. Solutions of type (1) are completely characterised in
Section~\ref{secL}, namely we prove that the cardinality of $X$ is a prime number $p$ and the group generated by $\lambda$ and $\rho$ is
cyclic of order $p$.

As for the solutions of type (2) we can naturally identify three subcases:
\begin{enumerate}
    \item $B^{(2)}=0$ and thus $V=[B,B]$,
    \item $B^{(2)}\ne 0$ and $B^{(3)}=0$ and thus $V=B^{(2)}$,
    \item $B^{(2)}=B^{(3)}\ne 0$.
\end{enumerate}

Let us examine these subcases. If $B^{(2)}=0$ this means that the additive and multiplicative structures of the skew braces coincide
and the associated solution is a derived solution. Hence, this reduces to the study of simple quandles, yielding the same classification
obtained by Joyce in \cite{MR0682881} and Andruskiewitsch and Graña in \cite{MR1994219}.

We achieve a complete classification of case (2) in Section~\ref{secTR}. To do so, we need to split the classification into two
cases based on the abelianity of $V$.
\begin{itemize}
    \item If $V$ is abelian we obtain the characterisation in Theorem~\ref{coro1*}. In summary, we prove that $V$ must be an
    elementary abelian $p$-group and that the additive structure of $B$ is given by the semi-direct product of $V$ and a cyclic
    group of order $k$. We also provide a concrete expression for the lambda map of  $B$, which allows us to derive an explicit
    formula for the simple solution obtained from this skew brace.
    \item If $V$ is non-abelian we obtain the characterisation in Theorem~\ref{coro2*}. Here we prove that $V$ must be a finite
    non-abelian characteristically simple group and the additive group of $B$ is the quotient of a semi-direct product of $V$ with
    a cyclic group over the center. In this case as well, we provide an explicit formula for the simple solution obtained from this skew brace. 
\end{itemize}
 
To summarise we have complete control over the case when $B^{(3)}=0$. It is also worth mentioning that all solutions obtained within
this classification are neither involutive nor derived (since we are assuming $B^{(2)}\ne 0$). Hence, they provide an abundance of new
simple solutions complementary to all the work done in the context of involutive simple solutions (see \cite{MR4161288,MR4300920,CeOkn,CO2022}).

Case (3) is open, including simple solutions that have the associated skew brace simple and, more generally, perfect. As one might expect,
we do not have a complete classification. However, our theorem can be applied to existing families of simple skew braces to obtain families
of solutions, as shown in Example~\ref{ex:byott}, where we construct an infinite family of simple solutions based on the family of simple
skew braces constructed by Byott in \cite{Byott-IP}.

\section{Preliminaries}

Let $(X,r)$ be a finite non-degenerate solution and $(X,r')$ its (left) derived solution.
We denote by \[M=M(X,r)=\free{X\mid x\circ y=\lambda_x(y)\circ\rho_y(x)\text{ for all }x,y\in X}\]
and \[A=A(X,r)=\free{X\mid x+y=y+\sigma_y(x)\text{ for all }x,y\in X}\] the \emph{Yang--Baxter monoid}
(also known as structure monoid of $(X,r)$) and the left derived Yang--Baxter monoid of $(X,r)$
(also known as derived structure monoid of $(X,r)$), respectively (note that $A(X,r)=M(X,r')$).
Moreover, \[G=G(X,r)=\gr(X\mid x\circ y=\lambda_x(y)\circ\rho_y(x)\text{ for all }x,y\in X)\]
denotes the \emph{Yang--Baxter group} (also known as structure group) of $(X,r)$ and
\[A_{\gr}=A_{\gr}(X,r)=\gr(X\mid x+y=y+\sigma_y(x)\text{ for all }x,y\in X)\]
denotes the \emph{left derived Yang--Baxter group} (also known as left derived structure group) of $(X,r)$
(here $\free{X\mid S}$ and $\gr(X\mid S)$ refer, respectively, to a monoid and group presentation with generators
from a set $X$ and relations from a set $S$). 

It was shown by Gateva-Ivanova and Majid \cite[Theorem 3.6]{GIMa08} that the maps $\lambda\colon X\to\Sym(X)\colon x\mapsto\lambda_x$
and $\rho\colon X\to\Sym(X)\colon x\mapsto \rho_x$ can be extended uniquely to $M=M(X,r)$ (for simplicity we use the same notation
for the extension) so that one obtains a monoid morphism $\lambda\colon(M,\circ)\to\Sym(M)\colon a\mapsto\lambda_a$, and a monoid
anti-morphism $\rho\colon(M,\circ)\to\Sym(M)\colon a\mapsto\rho_a$. Moreover, the map $r_M\colon M\times M\to M\times M$,
defined as $r_M(a,b)=(\lambda_a(b),\rho_b(a))$, is a non-degenerate solution of the Yang--Baxter equation.
Furthermore, for $a,b,c\in M$, we have:
\begin{align}
    a\circ b &=\lambda_a(b)\circ\rho_b(a),\label{eq:str1}\\
    \lambda_a(b\circ c) &=\lambda_a(b)\circ\lambda_{\rho_b(a)}(c),\label{eq:str2}\\
    \rho_a(b\circ c) &=\rho_{\lambda_c(a)}(b)\circ\rho_a(c).\label{eq:str3}
\end{align}
Also the map $\sigma\colon X\to\Sym(X)\colon x\mapsto \sigma_x$ can be extended to a monoid anti-morphism
$\sigma\colon A\to\Sym(M)$ given by $\sigma_b(a)=\lambda_b\rho_{\lambda^{-1}_a(b)}(a)$. One obtains in the
monoid $(A,+)$ that $a+b=b+\sigma_b(a)$ for all $a,b\in A$. It follows (see, e.g., \cite[Proposition 3.2]{CJV21})
that there exists a bijective $1$-cocycle $\pi\colon M\to A$ (via the $\lambda$-map, i.e., $a+\lambda_a(b)=a\circ b$
for $a,b\in A$) such that $\pi(x)=x$ for all $x\in X$. So, we have a monoid embedding 
\[f\colon M\to A\rtimes{\Img(\lambda)}\colon m\mapsto(\pi(m),\lambda_m).\] 
Abusing notation, we will often identify $m$ with $f(m)$ or $\pi(m)$. Hence, in particular, we may also denote by $x$,
with $x\in X$, the generators of $A$. So, we simply may write \[M=\{a=(a,\lambda_a):a\in A\}=\free{x=(x,\lambda_x):x\in X}.\]
Hence, in the terminology of \cite{CoJeVAVe21x}, $(M,+,\circ,\lambda,\sigma)$ is a  unital YB-semitruss, that is
$(M,+)$ and $(M,\circ)$ are monoids (with the same identity $0=1$), $\lambda\colon (M,\circ) \to \Aut (M,+)$
is a monoid morphism with \[a\circ(b+c)=a\circ b+\lambda_a (c)\quad\text{and}\quad a+\lambda_a(b)=a\circ b,\] and thus
\begin{equation}
    \lambda_a\lambda_b=\lambda_{\lambda_a(b)}\lambda_{\rho_b(a)},\label{eq:YBE1}
\end{equation}
$\sigma\colon (M,+)\to\Aut(M,+)$ is a monoid anti-morphism with 
\begin{equation}
    a+b=b+\sigma_b(a)\quad\text{and}\quad\sigma_{\lambda_a(b)}\lambda_a(c)=\lambda_a\sigma_b(c),\label{eq:lamsig}
\end{equation}
and thus also
\begin{equation}
    \sigma_a\sigma_b=\sigma_{\sigma_a(b)}\sigma_a\label{eq:sigsig}
\end{equation}
for all $a,b,c\in M$.

Because of the latter
assumption, one obtains that the maps $\lambda$, $\rho$, and $\sigma$ naturally induce left and right  actions on the structure group
(see also \cite{MR1178147}) and so that the equations \eqref{eq:str1}, \eqref{eq:str2}, and \eqref{eq:str3} remain valid in $G$;
we use the same notation for these maps. Moreover, the bijective $1$-cocycle $M\to A$ naturally induces a bijective $1$-cocycle
$G\to A_{\gr}$ (via the $\lambda$-map). Hence again we can identify $G$ with $A_{\gr}$ and thus $(G,+,\circ,\lambda,\sigma)$
is a YB-semitruss with $(G,+)$ and $(G,\circ)$ groups, in other words it is a skew left brace. Also,
\[G=\{(a,\lambda_a):a\in A_{\gr}\} \s A_{\gr}\rtimes {\Img(\lambda)}.\]
So, here we thus have \[\lambda_a (b)=-a+a\circ b\quad\text{and}\quad\sigma_b(a)=-b+a+b\] for all $a,b\in G$, and we have the
associated solution $(G,r_G)$ with $r_G(a,b)=(\lambda_a (b),\rho_b(a))$. Note that $X$ is naturally embedded in the structure
monoid $M$. However, the natural map $i\colon X\to G$ is not necessarily injective (see, e.g., \cite{MR3974961}) but it is a morphism
of solutions. A solution $(X,r)$ is said to be \emph{injective} if $i$ is injective. It is well known that involutive solutions
are always injective.

In \cite[Definition 1.5]{CJKVAV2020} the group, called the \emph{permutation group} of the non-degenerate solution $(X,r)$,
\[\mc{G}_{\lambda,\rho}=\mc{G}_{\lambda,\rho}(X,r)=\gr((\lambda_x,\rho_x^{-1}):x\in X)\s\Sym(X)^2,\]
has been introduced and it is shown that $\mc{G}_{\lambda,\rho}$ is equipped with a skew left brace structure so that the
natural map $G\to\mc{G}_{\lambda,\rho}$ is an epimorphism of skew left braces. We will consider an isomorphic copy of this
skew left brace. For this we first need to prove some lemmas. The following lemma has been shown partially in
\cite[Proposition 4.3]{CoJeVAVe21x}. However an identity shown in \cite[Lemma 3.3]{CoJeVAVe21x} gives an easier proof and
a complete symmetric result. Put \[\mf{q}\colon X\to X\colon x\mapsto\lambda_x^{-1}(x),\] known as the \emph{diagonal map}.
Recall that $\mf{q}$ is a bijective map because $(X,r)$ is non-degenerate and thus bijective (see \cite[Theorem 3.1]{CoJeVAVe21x},
\cite{MR4469332} and \cite{diagonals}).

\begin{lem}\label{lem1}
    Let $a,b\in A$ and assume $\{\alpha,\beta,\gamma\}=\{\lambda,\rho,\sigma\}$. If $\alpha_a=\alpha_b$ and $\beta_a=\beta_b$
    then also $\gamma_a=\gamma_b$. Furthermore, if $\alpha_a=\beta_a=\id$ then also $\gamma_a=\id$.
    \begin{proof}
        From \eqref{eq:lamsig} and \eqref{eq:YBE1}, and the definition of the $\sigma$-map, we get that
        $\rho_b(a)=\lambda^{-1}_{\lambda_a(b)}\sigma_{\lambda_a(b)}(a)=\lambda^{-1}_{\lambda_a(b)}\lambda_a\sigma_b\lambda^{-1}_a(a)
        =\lambda_{\rho_b(a)}\lambda^{-1}_b\sigma_b\mf{q}(a)$ for $a,b\in A$. Since, $\rho_b$ is bijective it yields 
        $\lambda_a\lambda^{-1}_b\sigma_b\mf{q}\rho^{-1}_b(a)=a$ for all $a\in A$.
        Hence, $\mf{q}(a)=\lambda^{-1}_b\sigma_b\mf{q}\rho^{-1}_b(a)$ for all $a\in A$.
        So, $\sigma_b=\lambda_b\mf{q}\rho_b\mf{q}^{-1}$. The result now is obvious.
    \end{proof}
\end{lem}

As a consequence of the previous lemma we can now extend \cite[Lemma 4.1]{CoJeVAVe21x} in the following sense.
For completeness' sake we include a proof.

\begin{lem}\label{lem2}
    Let $a,a',b,b'\in A$.
    \begin{enumerate}
        \item If $\sigma_a=\sigma_{a'}$ and $\lambda_a=\lambda_{a'}$ then $\lambda_{a\circ b}=\lambda_{a'\circ b}$, $\rho_{b\circ a}=\rho_{b\circ a'}$, 
        $\sigma_{b\circ a}=\sigma_{b\circ a'}$, and $\lambda_{a+b}=\lambda_{a'+b}$, $\rho_{a+b}=\rho_{a'+b}$, $\sigma_{a+b}=\sigma_{a'+b}$.
        \item If  $\sigma_b=\sigma_{b'}$ and $\lambda_b=\lambda_{b'}$ then $\lambda_{a\circ b}=\lambda_{a\circ b'}$, $\rho_{b\circ a}=\rho_{b'\circ a}$, 
        $\sigma_{b\circ a}=\sigma_{b'\circ a}$, and $\lambda_{a+b}=\lambda_{a+b'}$, $\rho_{a+b}=\rho_{a+b'}$, $\sigma_{a+b}=\sigma_{a+b'}$.
    \end{enumerate}
    \begin{proof}
        (1) Assume $\sigma_a=\sigma_{a'}$ and $\lambda_a=\lambda_{a'}$. Because of Lemma~\ref{lem1} we have that also $\rho_a=\rho_{a'}$.
        So, we have  that $\lambda_{a\circ b}=\lambda_a\lambda_b=\lambda_{a'}\lambda_b=\lambda_{a'\circ b}$ and similarly
        $\rho_{b\circ a}=\rho_{b\circ a'}$. Further,
        $\sigma_{\lambda_b(a)}=\lambda_b\sigma_a\lambda^{-1}_b=\lambda_b \sigma_{a'}\lambda^{-1}_b=\sigma_{\lambda_b(a')}$ yields
        $\sigma_{b\circ a}=\sigma_{b+\lambda_b(a)}=\sigma_{\lambda_b(a)}\sigma_b=\sigma_{\lambda_b(a')}\sigma_b=\sigma_{b\circ a'}$.
        Now, $\lambda_{a+b}=\lambda_{a\circ\lambda^{-1}_a(b)}=\lambda_a\lambda_{\lambda^{-1}_a(b)}
        =\lambda_{a'}\lambda_{\lambda^{-1}_{a'}(b)}=\lambda_{a'+b}$, and similarly $\rho_{a+b}=\rho_{a'+b}$.
        Moreover, $\sigma_{a+b}=\sigma_b\sigma_a=\sigma_b\sigma_{a'}=\sigma_{a'+b}$. 
        
        (2) Assume $\sigma_b=\sigma_{b'}$ and $\lambda_b=\lambda_{b'}$. Because of Lemma~\ref{lem1} we have that also $\rho_b=\rho_{b'}$.
        So, we have  $\lambda_{a\circ b}=\lambda_{a\circ b'}$, $\rho_{b\circ a}=\rho_{b'\circ a}$. Furthermore,
        $\sigma_{b\circ a}=\sigma_{b+\lambda_b(a)}=\sigma_{\lambda_b(a)}\sigma_b=\sigma_{\lambda_{b'}(a)}\sigma_{b'}=\sigma_{b'\circ a}$.
        Moreover, $a+b=b+\sigma_b(a)=b\circ\lambda_b \sigma_b(a)$. Then
        $\lambda_{a+b}=\lambda_{b\circ\lambda_b \sigma_b(a)}=\lambda_b\lambda_{\lambda_b \sigma_b(a)}
        =\lambda_{b'}\lambda_{\lambda_{b'} \sigma_{b'}(a)}=\lambda_{a+b'}$. Similarly $\rho_{a+b}=\rho_{a+b'}$. 
        Finally, $\sigma_{a+b}=\sigma_b\sigma_a=\sigma_{b'}\sigma_a=\sigma_{a+b'}$.
    \end{proof}
\end{lem}

We now introduce the set 
\[\mc{G}_{\sigma,\lambda}=\mc{G}_{\sigma,\lambda}(X,r)=\{\ga_a=(\sigma_a^{-1},\lambda_a):a\in A\}.\]
Because of Lemma~\ref{lem2} we have the following well-defined binary operation on this set:
\[\ga_a+\ga_b=\ga_{a+b}\quad\text{and}\quad\ga_a\circ\ga_b=\ga_{a\circ b}.\]
Clearly $\mc{G}_{\sigma, \lambda}$ is additively (and multiplicatively) generated by the set $\{\ga_x:x\in X\}$. As in
\cite[Section 4]{CoJeVAVe21x} one now obtains that $\mc{G}_{\sigma ,\lambda}$ is a skew left brace for these operations
with $\lambda$-map \[\lambda_{\ga_a}(\ga_b)=\ga_{\lambda_a(b)}\] and $\sigma$-map \[\sigma_{\ga_a}(\ga_b)=\ga_{\sigma_a(b)}.\]
(Note that we use a $\lambda$ notation in several contexts, for example for defining set-theoretical solutions and  on the skew
left brace $\mc{G}_{\sigma,\lambda}$. It is clear from the context and the index used, which $\lambda$-map we consider.
The same convention applies to the $\sigma$-map.)

Since the image of the map \[\ga\colon M=A\to\mc{G}_{\sigma,\lambda}\colon a\mapsto{\ga_a}\] 
consists of invertible elements, both additively and multiplicatively, it induces an epimorphism of skew left braces
(for which we use the same notation) \[\ga\colon G\to\mc{G}_{\sigma,\lambda}\colon a\mapsto{\ga_a}.\] 
Note that it also follows from Lemma~\ref{lem1} that the map $\mc{G}_{\sigma,\lambda}\to\mc{G}_{\lambda,\rho}$, defined as 
${\ga_a}\mapsto(\lambda_a,\rho_a^{-1})$, is an isomorphism of skew left braces. From now on we work with the skew left brace
$\mc{G}_{\sigma,\lambda}$ and we simply denote it as \[\mc{G}=\mc{G}_{\sigma,\lambda}.\] Recall that the \emph{socle} of the
skew left brace $G$ is by definition \[\Soc(G)=\Ker(\ga)=\Ker(\lambda)\cap Z(G,+)=\{a\in G:\lambda_a=\sigma_a=\id\}.\]
It is an ideal of the skew left brace $G$. Ideals of skew left braces are precisely the kernels of morphisms 
of skew left braces. It turns out that an additive subgroup $I$ of a skew left brace $(B,+,\circ)$ is an \emph{ideal} if and only
if it is $\lambda$-invariant, i.e., $\lambda_a (I)=I$ for all $a\in B$, $(I,+)$ is a normal subgroup of $(B,+)$ and $(I,\circ)$
is a normal subgroup of $(B,\circ)$. Equivalently, $I$ is $\lambda$-invariant, $(I,+)$ is a normal subgroup of $(B,+)$ and
$i*b\in I$ for all $i\in I$ and $b\in B$. Recall that the operation $*$ is defined as follows: \[a*b=-a+a\circ b-b=\lambda_a(b)-b\]
for all $a,b\in B$. An additive subgroup $I$ that is $\lambda$-invariant is said to be a \emph{left ideal} of the skew left brace.

The solution $(X,r)$ is said to be \emph{irretractable} \cite{MR3974961} (see also, e.g., \cite{ESS99,CJKVAV2020}) if the
restriction of the map $\ga\colon M\to\mc{G}_{\sigma,\lambda}$ to $X$ is injective. In case $\ga(X)$ is a singleton then
the solution $(X,r)$ is a Lyubashenko solution.

Clearly, the restriction of the map $\ga$ to $X$ is a morphism of solutions. Hence, if $(X,r)$ is simple then it follows that
either $\ga(X)$ is a singleton and we have a Lyubashenko solution, or the restriction of $\ga$ to $X$ is an injective map,
i.e., the solution $(X,r)$ is irretractable. Note that in the latter case $(X,r)$ is an injective solution and it is isomorphic
with the solution induced on $\ga(X)$. Furthermore, in the latter case $\mc{G}_{\sigma,\lambda}$ is additively
(and also multiplicatively) generated by the set $\ga(X)$.

Let \[\mc{C}=\{\sigma_a:a\in A\}.\]
It is a group for the multiplication \[\sigma_a\circ\sigma_b=\sigma_a\sigma_b=\sigma_{b+a}\] (note that the multiplication
$\circ$ in $\mc{C}$ coincides, in fact, with the ordinary composition of maps) and we consider this as a trivial skew left brace,
i.e., its additive structure is the same as its multiplicative structure (i.e., $\sigma_a + \sigma_b = \sigma_a\circ\sigma_b$).
The multiplicative group $\mc{G}$ acts on $\mc{C}$ via the action $\cdot$ defined as follows:
\[{\ga_a}\cdot{\sigma_b}=\sigma_{\lambda_a(b)}=\lambda_a\sigma_b\lambda_a^{-1}.\]
Hence we may consider the semi-direct product of skew left braces  \[\mc{S}=\mc{C}\rtimes \mc{G}.\] 
(For better readability we
will denote elements of $\mc{S}$ as pairs $(\sigma_a,\ga_b)$ with $a,b\in A$.) So, the additive operation is the componentwise
addition and the multiplication is as follows:
\[(\sigma_a,\ga_b)\circ(\sigma_c,\ga_d)=(\sigma_a\circ\sigma_{\lambda_b(c)},\ga_{b\circ d})=(\sigma_{\lambda_b(c)+a},\ga_{b\circ d}).\]

\begin{lem}\label{actionS}
    The group $(\mc{S},\circ)$ acts on $X$ as follows: \[(\sigma_a,\ga_b)\cdot x=\sigma_a \lambda_b(x).\]
    \begin{proof}
        By \eqref{eq:lamsig}, the mapping $(\mc{S}(X,r),\circ)\to(\Sym(X),\circ)$, defined as
        $(\sigma_a,\ga_b)\mapsto \sigma_a \lambda_b$, is a group homomorphism. Hence the statement follows.
    \end{proof}
\end{lem}

Note that the action of $\mc{S}$ on $X$ and the action of the group $\free{\lambda_x,\rho_y\colon x,y\in X}$ on $X$
have the same orbits (see for example the proof of \cite[Theorem 2.2]{MR4388351}).

\section{Simple solutions versus skew left braces}\label{sec:3}

In the skew left brace $\mc{S}=\mc{C}\rtimes \mc{G}$, as is common, we identify $\mc{C}$ and $\mc{G}$ as sub skew left braces.
In particular, we simply write $(\sigma_a,\ga_0)$ as $\sigma_a$ and $(\sigma_0,\ga_a)$ as $\ga_a$.

\begin{lem}\label{lemideal}
  Let $(X,r)$ be a finite non-degenerate solution of the Yang--Baxter equation. 
    Consider the additive subgroup  $\mc{D}=\free{\ga_x-\ga_y:x,y\in X\text{ and }}_+$ of $\mc{G}$ 
    and the additive subgroup 
    $\mc{D}'=\free{\sigma_x-\sigma_y:x,y\in X}_+$ of $\mc{C}$.
    Then,  $\mc{D}=\free{-\ga_x+\ga_y:x,y\in X}_+$ and
        $\mc{D}'=\free{\sigma_x^{-1}-\sigma_y^{-1}: x,y\in X}_+=\varphi(\mc{D})$. Moreover,
        $\mc{D}'$ and $\mb{D}=\mc{D}'\rtimes\mc{D}$ are ideal of $\mc{S}$.
    \begin{proof}
        Obviously, $\mc{D}$ is a left ideal of $\mc{G}$. To prove that $\mc{D}$ is an ideal of $\mc{G}$, we need to show that $(\mc{D},+)$
        and $(\mc{D},\circ)$ are normal subgroups of the additive, respectively, multiplicative group $\mc{G}$. The latter is equivalent with
        proving that $({\ga_x}-{\ga_y})*g\in \mc{D}$ for any $x,y\in X$ and $g\in \mc{G} $. It is sufficient to deal with $g=\ga_t$ with
        $t\in X$. Indeed, once we have shown that $(\mc{D},+)$ is a normal subgroup, it is sufficient to note that in any skew left brace $B$
        we have $a*(b+c)=a*b+b+a*c-b$ for all $a,b,c\in B$ (see, e.g., \cite{MR4256133}).
        
        We first prove that the subgroup $(\mc{D},+)$ is normal in $(\mc{G},+)$. If $x,y,t\in X$ then
        \[-\ga_t+(\ga_x-\ga_y)+\ga_t  =-{\ga_t}+\ga_t+\sigma_{\ga_t}(\ga_x-\ga_y)
        =\sigma_{\ga_t}(\ga_x)-\sigma_{\ga_t}(\ga_y)=\ga_{\sigma_t(x)}-\ga_{\sigma_t(y)}\in\mc{D}.\]
        Since this holds for all $t\in X$, it follows that indeed the additive subgroup is a normal subgroup.
        
        Also note that $-\ga_y+\ga_x=-{\ga_x}+(\ga_x-\ga_y)+\ga_x\in\mc{D}$ as $(\mc{D},+)$ is a normal subgroup of $(\mc{G},+)$.
        Hence, $\free{-\ga_x+\ga_y:x,y\in X}_+\s \mc{D}$. Similarly as above, one proves that
        $\free{-\ga_x+\ga_y:x,y\in X}_+$ is a normal subgroup, and that it contains $\mc{D}$.
        It follows that $\mc{D}=\free{-\ga_x+\ga_y:x,y\in X}_+$ and thus $\varphi(\mc{D})= \mc{D}'$.
        
        Next we prove that $(\mc{D},\circ)$ is a normal subgroup of $\mc{G}$. Because the diagonal map $\mf{q}$ is bijective,
        we write $y=\lambda_z^{-1}(z)=\lambda_{z^{-1}}(z)$ for some $z\in X$. Hence
        $\ga_y=\ga_{\lambda_{{z^{-1}}}(z)}=\lambda_{\ga_{z^{-1}}}(\ga_z)$. The latter means that
        $\ga_y=-\ga_{z^{-1}}+\ga_{z^{-1}}\circ \ga_z=-\ga_{z^{-1}}$, as $\ga_{z^{-1}}$ is the inverse of $\ga_z$ in $\mc{G}$
        for the multiplicative operation. Hence, $-\ga_y+\ga_x=\ga_{z^{-1}}+\ga_x=\ga_{z^{-1}} \circ \ga_{\lambda_z(x)}$,
        and thus \[(-\ga_y+\ga_x)*\ga_t  =(\ga_{z^{-1}}\circ \ga_{\lambda_z(x)})*\ga_t
        =\lambda_{\ga_{z^{-1}}\circ\ga_{\lambda_z(x)}}(\ga_t)-\ga_t=\ga_{\lambda_z^{-1}\lambda_{\lambda_z(x)}(t)}-\ga_t\in\mc{D}.\]
        
        Finally, we show that $\mb{D}$ and $\mc{D}'$ are ideals of $\mc{S}$. It  easily is verified that they are left ideals of $\mc{S}$.
        Since \[(\sigma_c,\ga_d)+(\sigma_a,\ga_b)-(\sigma_c,\ga_d)=(\sigma_{-c+a+c},\ga_{d+b-d})=(\sigma_{\sigma_c(a)},\ga_{d+b-d}),\]
        we get that for $a=y-x$,  the first coordinate is in $\mc{D}'$ and if $b\in \mc{D}$, then  the second coordinate belongs to $\mc{D}$
        (because $(\mc{D},+)$ is a normal subgroup of $(\mc{G},+)$). Hence, $(\mb{D},+)$ and $(\mc{D}',+)$ are  normal subgroups of $(\mc{S},+)$.
        Moreover,
        \begin{align*}
            (\sigma_c,\ga_d)^{-1} \circ \sigma_a\circ(\sigma_c,\ga_d) & =(\sigma_{\lambda^{-1}_d(-c)},\ga_{d^{-1}})\circ (\sigma_{c+a},\ga_d)\\
            & =\sigma_{\lambda_d^{-1}(c+a)+\lambda_d^{-1}(-c)}=\sigma_{\lambda_d^{-1}(-c)}\sigma_{\lambda_d^{-1}(a)}\sigma_{\lambda_d^{-1}(c)}.
        \end{align*}
        So, again, considering $a=y-x$, we obtain that $\mc{D}'$  also is a multiplicative normal subgroup.
        In general, using \eqref{eq:sigsig}, we have
        \begin{align*}
            (\sigma_c,\ga_d)^{-1}\circ(\sigma_a,\ga_b) \circ(\sigma_c,\ga_d)
            &=(\sigma_{\lambda^{-1}_d(-c)},\ga_{d^{-1}})\circ (\sigma_{\lambda_b(c)+a},\ga_{b\circ d})
            =(\sigma_{\lambda_d^{-1}(\lambda_b(c)+a)+\lambda_d^{-1}(-c)},\ga_{d^{-1}\circ b\circ d})\\
            &=(\sigma_{\lambda_d^{-1}(-c)}\sigma_{\lambda_d^{-1}(a)}\sigma_{\lambda_d^{-1}(\lambda_b(c))},\ga_{d^{-1}\circ b \circ d})
            =(\sigma_u\sigma_{\lambda_d^{-1}(a)}\sigma_v,\ga_{d^{-1}\circ b\circ d})\\
            &=(\sigma_{\sigma_u(\lambda_d^{-1}(a))}\sigma_u\sigma_v,\ga_{d^{-1}\circ b\circ d})
            =(\sigma_{\sigma_u(\lambda_d^{-1}(a))}\sigma_u\sigma_v,\ga_d^{-1}\circ\ga_b\circ\ga_d),
        \end{align*}
        where $u=\lambda_d^{-1}(-c)$ and $v=\lambda_d^{-1}(\lambda_b(c))$. So, if $a=x-y$, then, as before,
        $\sigma_{u}\sigma_{v}\sigma_{\sigma_{v}(\lambda_d^{-1}(a))}\in \mc{D}'$ and if $b$ is also a difference then we know
        that the second term  also is  in $\mc{D}$. This shows that $\mb{D}$ is a multiplicative normal subgroup of $\mc{S}$.
        Hence, $\mb{D}$ is indeed an ideal of the skew left brace $\mc{S}$.       
    \end{proof}
\end{lem}

\begin{thm}\label{simpleNL}
    Let $(X,r)$ be a finite non-degenerate solution of the Yang--Baxter equation with $|X|>1$.
    If $(X,r)$ is not a Lyubashenko solution, then $(X,r)$ is simple if and only if $(X,r)$ is irretractable,
    the ideal $\mc{D}$  is the smallest non-zero ideal of $\mc{G}$ and the group $\mb{D}$ acts transitively on $X$.
    \begin{proof}
        Because $(X,r)$ is not a Lyubashenko solution, it is irretractable, i.e., the map $X\to \mc{G}\colon x\mapsto\ga_x$
        is injective. Therefore, the solution $(X,r)$ is the restriction of the solution $r_{\mc{G}}$ associated to $\mc{G}=\mc{G}(X,r)$.
        Let $I$ be non-zero ideal of $\mc{G}$ and let $Y$ be the natural image of $X$ in the skew left brace $\mc{G}/I$.
        Clearly $Y$ inherits a solution of $X$, actually the solution on $Y$ is a restriction of the solution associated to
        $\mc{G}/I$. We thus get an epimorphism of solutions $f\colon X\to Y\colon x\mapsto{\ga_x}+I$. Since $(X,r)$ is a simple
        solution we have that $|Y|=1$ or $f$ is an isomorphism. Suppose the latter holds. If $0\ne\ga_a\in I$ then
        $\ga_a*\ga_x=\lambda_{\ga_a}(\ga_x)-\ga_x=\ga_{\lambda_a(x)}-\ga_x\in I$ and thus ${\ga_x}+I={\ga_{\lambda_a(x)}}+I$
        for any $x\in X$. Since $f$ is an isomorphism it follows that $\lambda_a=\id$. Since $\ga_a\ne 0$, we get $\sigma_a\ne\id$.
        Therefore, there exists $x\in X$ such that $\sigma_a(x)\ne x$. Then
        \[[-\ga_x,-\ga_a]_+=-\ga_x-\ga_a+\ga_x+\ga_a=-{\ga_x}+\sigma_{\ga_a}(\ga_x)=-\ga_x+\ga_{\sigma_a(x)}.\]
        Since $[-\ga_x,-\ga_a]_+\in I$, we obtain ${\ga_x}+I={\ga_{\sigma_a(x)}}+I$, which contradicts the bijectivity of $f$.
        Hence, $|Y|=1$ and so ${\ga_x}+I ={\ga_y}+I$ for all $x,y\in X$. Thus $I$ contains all elements $\ga_x-\ga_y$ and so
        $\mc{D}\s I$. Hence, we have shown that $\mc{D}$ is the smallest non-zero ideal in $\mc{G}$. To prove that the action of
        the group $\mb{D}$ on $X$ is transitive consider $\mb{D}$ as an ideal of $\mc{S}$ (see Lemma~\ref{lemideal})
        and consider the natural epimorphism of solutions $X\to X/{\sim_\mb{D}}$. Because $(X,r)$ is a simple solution, we get
        that $|X/{\sim_\mb{D}}|=1$. Thus indeed $\mb{D}$ acts transitively on $X$. This proves the necessity of the conditions.
        
        To prove the necessity of the conditions, assume $f\colon(X,r)\to(Y,s)$ is a non-bijective epimorphism of solutions.
        By the assumptions $X$ is a subset of $\mc{G}=\mc{G}(X,r)$ and $r$ is the restriction of $r_{\mc{G}}$ to $X$.
        Let $\mc{G}(f)\colon\mc{G}\to \mc{G}(Y,s)$ be the induced epimorphism of skew left left braces. Since $f$ is not
        injective also $\mc{G}(f)$ is not injective and thus $\Ker(\mc{G}(f))\ne 0$. Hence, by the assumptions,
        $\mc{D}=\langle \ga_x-\ga_y:x,y\in X\rangle_{+}\s\Ker(\mc{G}(f))$. Again, by the assumptions, $\mb{D}=\mb{D}(X,r)$
        acts transitively on $X$. So, it follows that $\mb{D}\cdot x=X$ for $x\in X$. Since $\mc{D}\s\Ker(\mc{G}(f))$
        we have that $Y=f(X)=\{ f(x)\}$, i.e., $|Y|=1$, as desired. 
    \end{proof}
\end{thm}

We now formulate the main result in what appears at first sight a more general context. However the following reasoning shows this is
actually not the case. Let $B$ be a finite skew left brace and let $(B,r_B)$ be its associated solution. Hence we have the associated
skew left brace $\mc{S}(B)=\mc{S}(B,r_B)$ that acts on the set $B$ via the group $(B,+)\rtimes (B,\circ)$. In other words, the element
$(a,b)$, with $a,b\in B$, acts on $x\in B$ as follows $(a,b)\cdot x=\sigma_a\lambda_b(x)$. If $X$ is a subset of $B$ and $I$ is
an ideal of $B$ then we say that $I$ acts on $X$ if $X$ is invariant under the action of the subgroup $I\rtimes I$ of $(B,+)\rtimes (B,\circ)$.

Assume $X$ is a subset of $B$ that is an additive set of generators of $(B,+)$. Assume  furthermore that $r_B$ restricts to a solution
$r=r_B|_{X\times X}$ on $X$ (i.e., $X$ is invariant under the action of $\mc{S}(B)$). Hence, we obtain natural morphisms of skew
left braces \[G(X,r)\to G(B,r_B)\to B\to \mc{G}(B,r_B)\to \mc{G}(X,r)\] (the first morphism is induced by the natural embedding
$(X,r)\to(B,r_B)$, the second follows because the morphism from the free group on $B$ to $B$ respects the defining relations of
$G(B,r_B)$, the third one is $a\mapsto(\sigma^{-1}_a,\lambda_a)$, and the last one is the restriction). Since the composition of all
these maps is the natural epimorphism $\ga\colon G(X,r)\to\mc{G}(X,r)$, the composed map $B\to\mc{G}(B,r_B) \to \mc{G}(X,r)$ is an epimorphism.
If also $(X,r)$ is not a Lyubashenko solution then there exist distinct $x,y\in X$ with different images in $\mc{G}(X,r)$. Hence if the ideal $V$
generated by all $x-y$, with $x,y\in X\s B$, is the smallest non-zero ideal of $B$ then it follows that $B\cong \mc{G}(B,r_B)\cong\mc{G}(X,r)$.

So, we obtain the following result.

\begin{thm}\label{simpleGEN}
    Let $(X,r)$ be finite non-degenerate solution of the Yang--Baxter equation that is not of Lyubashenko type.
    Then $(X,r)$ is simple if and only if $(X,r)$ is embedded into the associated solution $(B,r_B)$ of a finite skew left
    brace $B$ such that $B$ is additively generated by $X$, the ideal $V$ generated by the set $\{x-y:x,y\in X\}$
    is the smallest non-zero ideal of $B$, and the action of $V$ as an ideal of $B$ on $X$ is transitive.
\end{thm}

If in the previous theorem $B$ is a finite simple skew left brace then the conditions are equivalent with $B$ is additively generated
by $X$ and the action of the group $\free{\sigma_a\lambda_b:a,b\in B}$ on $X$ is transitive. It is worth mentioning that in
\cite[Remark 5.3]{CeOkn} examples of finite simple skew left braces of abelian type $B$ are given (so the additive structure
is commutative) such that for every orbit $\Omega$ of any element under the action of the $\lambda$-map, $B\ne\free{\Omega}_+$.

Assume $(B,+,\circ)$ and $X$  are as in Theorem~\ref{simpleGEN} and $V\ne B$. Then $B/V$ is a trivial skew left brace of cyclic type.
In particular, $B^{(2)}=B*B\s V$. As $B^{(2)}$ is an ideal of $B$, and, by assumption $V$ is the smallest non-zero ideal of $B$, it follows
that either $B^{(2)}=0$ or $B^{(2)}=V$. The former corresponds to the case where $B$ is a trivial skew left brace, and thus $(X,r)$ is a
left derived solution. In the latter case, $B/B^{(2)}$ is a trivial skew left brace of cyclic type and $B^{(2)}$ acts transitively on $X$.

If $B^{(2)}\ne 0$ then, since $B^{(3)}=B^{(2)}*B$ is an ideal of $B$, we get that either $B^{(3)}=0$ or $B^{(3)}=B^{(2)}$.

Assume, on the other hand, that  $B^{(2)}=0$. Note that then  $B$ cannot be of abelian type if $|X|>2$ (otherwise, the solution is the
twist map and it is simple if and only if $|X|\le 2$). Since by Theorem~\ref{simpleNL} $(X,r)$ is irretractable, it also follows
that $Z(B)$ is trivial. Indeed, if $a\in Z(B)$ and $y\in X$, then $\lambda_a=\id$ and $\sigma_a(y)\in X$ and $\sigma_a(y)=y$. So,
$\sigma_a=\id$ and $\lambda_a=\id$, and since the map $B\to \mathcal{G}(X,r)$, given as $a\mapsto(\sigma_a^{-1},\lambda_a)$, is an
isomorphism (see the comment just before Theorem \ref{simpleGEN}), we obtain that $a=0$. Moreover, the smallest non-zero ideal $V$
needs to coincide with the derived subgroup $[B,B]$. This follows from the fact that $X$ is an orbit and so all generating elements
in $V$ are commutators. Indeed, let $x-y$ be a generator of $V$, since $X$ is an orbit there exists $a\in B$ such that
$y=\sigma_a(x)=-a+x+a$, and thus $x-y=x-a-x+a=[x,-a]$. Finally, note that if in this case also $[B,B]\ne B$ then, since $X$
is a conjugacy class generating $B$, the group  $B/[B,B]$ needs to be cyclic. Note that this recovers the classification
of simple quandles given by Joyce in \cite{MR0682881} and which is stated in  part two of the following corollary.
Hence we have shown Theorem~\ref{CorSummary} stated in the introduction.
 
We finish this section with a remark on the derived solution of a simple solution.

\begin{rem}
    Clearly, since there are finite simple non-degenerate involutive solutions $(X,r)$ of cardinality larger than $2$,
    the left derived solution $(X,r')$ does not have to be simple if $(X,r)$ is simple. The converse however holds,
    as already noticed in \cite{MR1994219}.

    Indeed, assume $f\colon(X,r)\to (Y,s)$ is an epimorphism of finite non-degenerate solutions.
    Then one easily verifies that it induces an epimorphism $(X,r')\to (Y,s')$ between the  left derived
    solution $(X,r')$ of $(X,r)$ and the left derived solution $(Y,s')$ of $(Y,s)$. If $(X,r')$ is a simple
    solution, then we get that $|Y|=|X|$ (and thus $f$ is bijective) or $|Y|=1$. So, a finite non-degenerate
    solution $(X,r)$ is simple if its derived solution $(X,r')$ is simple.
\end{rem}

\section{Lyubashenko solutions}\label{secL}

In this section we describe the finite simple non-degenerate solutions that are of Lyubashenko type,
that is solutions of type (1) in Theorem~\ref{CorSummary}. In \cite{MR1848966} Etingof, Guralnick, and
Soloviev have described the indecomposable non-degenerate solutions of prime order. We now show that
Lyubashenko solutions that are simple have to be of prime order and indecomposable.

Recall that a solution $(X,r)$ is said to be \emph{primitive} if the group $\mc{S}(X,r)$ acts primitively on $X$ (cf. \cite{MR4388351}).
One may check that for a Lyubashenko solution this is the same as saying that $\free{\lambda,\rho}$ is a primitive permutation group.

\begin{lem}\label{primitive1}
    Let $(X,r)$ be a finite non-degenerate Lyubashenko solution. If $(X,r)$ is simple then $(X,r)$ is primitive.
\end{lem}
\begin{proof}
    Since $(X,r)$ is simple, it is indecomposable. So, the group $\mc{S}(X,r)$ acts transitively on $X$. Assume that $X=X_1\cup\dotsb\cup X_n$,
    where $X_1,\dotsc,X_n$ are imprimitive blocks. We need to prove that either $n=1$ or $|X_1|=\dotsb=|X_n|=1$. To do so, consider $\ov{X}$
    the set of all imprimitive blocks. Because $(X,r)$ is a Lyubashenko solution, we get an induced solution $\ov{r}$ on $\ov{X}$ and a natural
    surjective morphism of solutions $(X,r)\to(\ov{X},\ov{r})$. Hence, since $(X,r)$ is simple, we conclude that either $\ov{X}$ trivial, i.e.,
    $n=1$ or $(\ov{X},\ov{r})$ is isomorphic to $(X,r)$, i.e., $|X_1|=\dotsb=|X_n|=1$.
\end{proof}

\begin{cor}\label{SimpleLYU}
    Assume $(X,r)$ is a finite non-degenerate Lyubashenko solution with $|X|\ge 3$, i.e., $r(x,y)=(\lambda(y),\rho(x))$
    for some commuting permutations $\lambda,\rho\in\Sym(X)$. Then $(X,r)$ is simple if and only if $|X|=p$
    for some prime $p$ and the group $\free{\lambda,\rho}$ is cyclic of order $p$.
    
    In other words, identifying $X$ with $\mb{Z}_{p}$, such solutions are isomorphic with solutions of the form
    $\mb{Z}_p^2\to \mb{Z}_p^2\colon(x,y)\mapsto(y+a,x+b) $ with $a,b\in \mb{Z}_{p}$ and $(a,b)\ne (0,0)$.
    \begin{proof}
        By Lemma~\ref{primitive1}, $H=\free{\lambda,\rho}$ is a primitive abelian permutation group on $X$. Hence, it is
        known (see for instance \cite[Section 7.2]{MR1357169}) that $H$ is  of prime order $p$, and the statement follows.
    \end{proof}
\end{cor}

Note that there are examples of non-involutive simple non-degenerate solutions of prime order that are not of Lyubashenko type.
This is in contrast with the involutive case where there is only one involutive simple solution of prime order (see \cite{ESS99})
and it is of Lyubashenko type, i.e., of the form $r(x,y)=(\lambda(y),\lambda^{-1}(x))$ with $\lambda$ a full cycle on $X$.

\begin{ex}
    Let  $X=\{1,2,3\}$ and define $\rho_1=(2,3)$, $\rho_2=(1,3)$ and $\rho_3=(1,2)$ (this is the conjugation action in $\Sym(3)$
    if one identifies $X$ with the set of elements of order $2$ in $\Sym(3)$). Then the solution $r(x,y)=(y,\rho_y(x))$ is
    simple by Example~\ref{exsemipri}. Note that since all $\lambda_x=\id$ we have that $\sigma_x=\rho_x$ for each $x\in X$.
    Hence, for this example, $\mc{G}=\{(\rho_a^{-1},\id):a\in A\}\cong \Sym(3)$ and $\mc{C}=\{\rho_a:a\in A\}\cong \Sym(3)$.
    Clearly this example can be generalized by making the size of $X$ equal to any prime $p>2$.
\end{ex}

\section{Solutions with \texorpdfstring{$B^{(3)}=0$}{B(3)=0}}\label{secTR}

In this section we shall focus on finite simple non-degenerate solutions $(X,r)$ which are of type (2) in Theorem~\ref{CorSummary}.
So, in particular, these are solutions that are embedded in a solution determined by a finite skew left brace $B$ that has a unique
minimal non-zero ideal $V$ (recall that it necessarily contains all elements $x-y$ with $x,y\in X$) such that $B/V$ is a trivial
skew left brace of cyclic type and $B^{(3)}=0$.

Let us start with some elementary observations.

\begin{rem}\label{elform}
    Since $B$ is additively generated by $X$ and $x-y\in V$ for any $x,y\in X$, we may write $B=V+\free{x}_+$ for some (any) $x\in X$.
    In particular, if $k=|B/V|$ then each element of $B$ can be uniquely written as $v+ix$ for some $v\in V$ and $0\le i<k$. Note
    that $B\ne B^{(2)}$ implies $V\ne B$, and thus $k>1$.
\end{rem}

The condition $B^{(3)}=0$ implies the the solution $(X,r)$ is of special type.

\begin{rem}\label{lambdaprop}
    Because $V*B=0$, we get $\lambda_v=\id$ for each $v\in V$. In particular, since $kx\in V$, we have $\lambda_x^k=\lambda_{kx}=\id$.
    Moreover, as $x-y\in V$ for any $x,y\in X$, we obtain $\lambda_x=\lambda_{(x-y)+y}=\lambda_{(x-y)\circ y}=\lambda_{x-y}\lambda_y=\lambda_y$.
    Therefore, $\lambda_x=\lambda_y$ for any $x,y\in X$. In what follows we shall simply denote this map as $\lambda$. In particular, $(X,r)$
    is the so-called twisted rack solution (as defined in \cite{CJKVA2023}), i.e., $r(x,y)=(\lambda(y),\rho_y(x))$ for all $x,y\in X$.
    Also recall that if $B^{(2)}=0$ then $\lambda=\id$ and $(X,r)$ is a rack solution.
\end{rem}

Note that the converse of Remark~\ref{lambdaprop} is also true. Namely, if $(X,r)$ is a twisted rack solution, which
is a subsolution of a solution associated to a skew left brace $B$ with the listed properties, then $B^{(3)}=0$. Indeed,
in this case $\lambda_a=\lambda^{|a|}$ for $a\in B$, as $B$ is generated by the set $X$ (here $|a|$ means the additive order
of the image of $a$ in $B/V$). In particular, $\lambda_{y-x}=\id$ for any $x,y\in X$, so $(y-x)*B=0$. Then, since
$(a+b)*c=a*(\lambda_a^{-1}(b)*c)+\lambda_a^{-1}(b)*c+a*c$ for all $a,b,c\in B$, it follows that $V*B=0$ and thus $B^{(3)}=B^{(2)}*B=0$.

The equality $B^{(3)}=0$ implies also that $V$ is a trivial skew left brace, i.e., the additive and multiplicative operations on $V$
coincide and thus $V$ is simply a group. It turns out that this group is very special.

\begin{lem}\label{charsimp}
    The group $V$ is characteristically simple, i.e., $V=S_1+\dotsb+S_n$, a direct sum of pairwise isomorphic simple groups.
    In particular, if $V$ is abelian then $V\cong C_p^n$ is an elementary abelian $p$-group for some prime $p$, whereas if
    $V$ is non-abelian then $Z(V)=0$.
    \begin{proof}
        Let $N_1,\dotsc,N_m$ denote representatives of the isomorphism classes of minimal normal subgroups of $V$. Let $I_i$,
        for $1\le i\le m$, denote the sum of all minimal normal subgroups of $V$ isomorphic to $N_i$. Since $I_i$ is a characteristic
        subgroup of $V$ and $V$ is an ideal of $B$, we conclude that $I_i$ is an ideal of $B$. Since $V$ is the smallest non-zero ideal
        of $B$, it follows that there exists just one such `isotypic component', that is $m=1$ and $V=I_1$. Now, write $V$ as a sum
        $V=S_1+\dotsb+S_n$ of minimal normal subgroups which are all isomorphic (to $N_1$). If none of the $S_i$ is redundant, i.e.,
        $S_i$ is not contained in $\sum_{j\ne i}S_j$, then the sum $S_1+\dotsb+S_n$ is direct. The remaining part of lemma is obvious.
    \end{proof}
\end{lem}

Further, two particular automorphisms of $V$ will play a crucial role. Namely, as $V$ is invariant with respect to
conjugation by $x$ and $\lambda$, it follows that there exist automorphisms $A,A'\in\Aut(V)$ such that
\[x+v-x=Av\quad\text{and}\quad\lambda(v)=A'v\] for each $v\in V$ (i.e., $A$ and $A'$ are restrictions of $\sigma_{-x}$
and $\lambda$, respectively, to $V$). Moreover, let \[\mc{A}=\free{A,A'}\] denote the subgroup of $\Aut(V)$ generated
by $A$ and $A'$.

\begin{lem}\label{Asim}
    $V$ has no non-trivial $\mc{A}$-invariant normal subgroups.
    \begin{proof}
        Assume that $I\ne 0$ is an $\mc{A}$-invariant normal subgroup of $V$. First, note that if $w\in I$ then
        $(v+ix)+w-(v+ix)=v+A^iw-v\in I$ for each $v\in V$ and $0\le i<k$. Hence, $I$ is a normal subgroup of $(B,+)$.
        Clearly, $I$ is also $\lambda$-invariant. Finally, $I*B\s V*B=0$ assures that $I$ is non-zero ideal contained in $V$, and thus $I=V$.
    \end{proof}
\end{lem}

Also the following observation will be useful.

\begin{lem}\label{Bprop}
    Let $a=v+ix\in B$ for some $v\in V$ and $0\le i<k$.
    \begin{enumerate}
        \item $a\in Z(B,+)$ if and only if $Av=v$ and $A^i=\sigma_v$ on $V$.
        \item $V\cap Z(B,+)=0$.
        \item $\Soc(B)=0$.
    \end{enumerate}
    \begin{proof}
        If $a\in Z(B,+)$ then $v+ix=a=x+a-x=x+v-x+ix=Av+ix$ implies $Av=v$. Further, if $w\in V$ then the equality $w+v+ix=w+a=a+w=v+ix+w=v+A^iw+ix$
        assures that $A^iw=-v+w+v=\sigma_v(w)$. Conversely, assume that $Av=v$ and $A^i=\sigma_v$ on $V$. Then $x+a-x=x+v-x+ix=Av+ix=v+ix=a$ and
        \[w+a-w  =w+v+ix-w=w+v-A^iw+ix=w+v-(-v+w+v)+ix=v+ix=a\] for each $w\in V$. Hence $a\in Z(B,+)$, and (1) is proved.

        Next, if $V$ is abelian then $V\cap Z(B,+)$ is an ideal of $B$ contained in $V$. So, it must be equal to zero because
        otherwise we would get $V\cap Z(B,+)=V$ and thus $V\s Z(B,+)$. Hence, the group $(B,+)$ would be abelian, as $B/V$
        is a cyclic group. Therefore, $X$ would be a singleton because the action of $V$ on $X$ is transitive,
        a contradiction. Whereas if $V$ is non-abelian then $Z(V)=0$ by Lemma~\ref{charsimp} and thus $V\cap Z(B,+)\s Z(V)=0$,
        which finishes the proof of (2).

        Finally, since $\Soc(B)$ is an ideal of $B$ and $V\cap\Soc(B)\s V\cap Z(B,+)=0$ by (2), it follows that $\Soc(B)=0$
        because otherwise $\Soc(B)$ would be a non-zero ideal not containing $V$, a contradiction. Thus (3) follows,
        and the proof is complete.
    \end{proof}
\end{lem}

Next, notice that the fourth condition in part (2) of Theorem~\ref{CorSummary} says that the action of $V$ on $X$ is transitive.
Since $\lambda_v=\id$ for all $v\in V$, this means that the action of $\mc{C}=\free{\sigma_v:v\in V}$ on $X$ is transitive. That is,
the (additive) conjugation action of $V$ on $X$ is transitive. So, \[X=\mc{C}_{(B,+)}(x)=\mc{C}_{(V,+)}(x)=\{v+x-v:v\in V\}=\{v-Av+x:v\in V\}.\]

Further, since $X$ is $\lambda$-invariant, we get \[\lambda(x)=u_0+x-u_0\] for some $u_0\in V$ (note that the element $u_0$ is in general
not uniquely determined, but since $u_0-Au_0=\lambda(x)-x=\lambda_x(x)-x=x*x$, the element $u_0$ is unique as long as
$A$ has no non-zero fixed points).

Define \[v_n=u_0-A^nu_0\in V\] for $n\ge 0$. Note that $v_0=0$ and $v_n=v_1+Av_1+\dotsb+A^{n-1}v_1$ for $n\ge 1$. Moreover, $\lambda(x)=v_1+x$.
Now, if $v\in V$ and $0\le i<k$ then using the above notation we can write
\begin{align}\label{lam}
    \begin{aligned}
        \lambda(v+ix) & =\lambda(v)+i\lambda(x)=A'v+i(u_0+x-u_0)
          =A'v+u_0+ix-u_0\\
          &=A'v+u_0-A^iu_0+ix=A'v+v_i+ix.
    \end{aligned}
\end{align}
Furthermore, 
\begin{align*}
    v_1+AA'v+x & =v_1+x+A'v=\lambda(x)+\lambda(v)=\lambda(x+v)\\
     & =\lambda(Av+x)=\lambda(Av)+\lambda(x)=A'Av+v_1+x.
\end{align*}
Therefore, $v_1+AA'v=A'Av+v_1$. So, replacing $v$ by $A^{-1}A'^{-1}v$ we may rewrite the last equality as $AA'A^{-1}A'^{-1}v=-v_1+v+v_1$.
Finally, since $v\in V$ is arbitrary the latter equality is equivalent to
\begin{equation}
    [A,A']=\sigma_{v_1}\text{ on }V.\label{C2}
\end{equation}

Our further investigation naturally splits into two cases (depending whether $V$ is abelian or non-abelian),
which nevertheless shares some similarities. We first treat the abelian case.

\subsection{The abelian case}

So, assume that the group $V$ is abelian. In this case Lemma~\ref{charsimp} implies that $V\cong C_p^n$. Moreover, since $kx\in V$,
we get $kx\in V\cap Z(B,+)=0$ by Lemma~\ref{Bprop}(2). Hence, $\free{x}_+\cong C_k$ and $(B,+)\cong V\rtimes\free{x}_+\cong C_p^n\rtimes C_k$.

Moreover, we have $A-1\in\Aut(V)$ (here and later we denote the identity map $\id_V$ of $V$ as $1$, for simplicity).
Indeed, if $Av=v$ for some $0\ne v\in V$ then $x+v-x=Av=v$, and thus $v\in V\cap Z(B,+)=0$ by Lemma~\ref{Bprop}(3), a contradiction. 

Further, $kx=0$ implies that $A^k=1\quad\text{and}\quad A'^k=1$, and thus the (multiplicative) orders
$o(A)$ and $o(A')$ of $A$ and $A'$, respectively, divide $k$. Furthermore, we claim that 
\begin{equation}\label{OrdCond}
    A^i\ne 1\text{ or }A'^i\ne 1\text{ or }1+A'+\dotsb+A'^{i-1}\ne 0\text{ for each }0<i<k.
\end{equation}
Indeed, suppose $A^i=A'^i=1$ and $1+A'+\dotsb+A'^{i-1}=0$ for some $0<i<k$. Then $ix\in Z(B,+)$ by Lemma~\ref{Bprop}(1).
Because of \eqref{lam} we have $\lambda(ix)=u_0-A^iu_0+ix=ix$. In particular, the subgroup $I=\free{ix}_+\s Z(B,+)$
is $\lambda$-invariant. Next, by an easy induction argument and because $1+A'+\dotsb+A'^{i-1}=0$, we obtain
  \[
    \lambda^i(x)  =\lambda^{i-1}(v_1+x)=A'^{i-1}v_1+\lambda^{i-1}(x)=A'^{i-1}v_1+\dotsb+A'v_1+v_1+x=x.\]
Hence, $\lambda^i(v+jx)=\lambda^i(v)+j\lambda^i(x)=A'^iv+jx=v+jx$ for all $v\in V$ and $0\le j<k$.
Therefore, $\lambda^i=\id$ and, in consequence, $(ix)*b=\lambda_{ix}(b)-b=\lambda^i(b)-b=0$ for each $b\in B$,
i.e., $I*B=0$. Summarising, $I$ is a non-zero ideal of $B$, which clearly does not contain $V$, a contradiction.

Finally, since $\sigma_{v_1}=\id$ on $V$, as $V$ is abelian, the condition \eqref{C2} yields $AA'=A'A$. In particular,
the group $\mc{A}=\free{A,A'}$ is abelian. Furthermore, if $\mc{R}$ denotes the subring of the endomorphism ring $\End(V)$
generated by $A$ and $A'$ then Lemma~\ref{Asim} assures that $V$ is a simple $\mc{R}$-module. A useful remark is that,
since $V*B=0$, an ideal of $B$ contained in $V$ is a synonym for $\mc{R}$-submodule of $V$.

\begin{rem}\label{primitive}
    \phantom{}
    \begin{enumerate}
        \item Because the simple $\mc{R}$-module $V$ is faithful, the commutative ring $\mc{R}$ is primitive. Hence, $\mc{R}$ is
        a field extension of $\mb{F}_p$, the field with $p$ elements. Moreover, $V\cong\mc{R}$ as $\mc{R}$-vector spaces,
        and thus $|\mc{R}|=|V|=p^n$, which yields $\mc{R}\cong\mb{F}_{p^n}$, the field with $p^n$ elements. In particular,
        the multiplicative orders $o(A)$ and $o(A')$ are divisors of $p^n-1$ (and thus both orders are not divisible by $p$),
        and $\lcm(o(A),o(A'))$ divides $\gcd(k,p^n-1)$. Furthermore, as $A\ne 1$, we must have $\gcd(k,p^n-1)\ne 1$.
        \item Write $k=p^sm$, where $s\ge 0$ and $p\nmid m$. If $C\in\mc{A}=\free{A,A'}$ then $C^k=1$ as $A^k=A'^k=1$
        and $AA'=A'A$. Hence, $0=C^k-1=(C^m-1)^{p^s}$ and, in consequence, $C^m=1$ as $\mc{R}$ is a field. In particular $o(C)$ is a
        divisor of $m$ (hence $o(C)$ is not divisible by $p$) and $|\mc{A}|\le m^2$. Moreover, $A\ne 1$ implies
        $1+A+\dotsb+A^{m-1}=0$ and thus, by (1), we get \[n=\dim_{\mb{F}_p}\mc{R}\le m(m-1)\le k(k-1).\] Note that
        if $\mc{A}=\free{A}$ (for example, if $A=A'$ and $\lambda$ is an additive conjugation by $x$) then we get a
        stronger bound $n<m\le k$. Further, $k$ cannot be a $p$-power as $m>1$.
        \item If $A'=1$ and $v_1=0$ (i.e., $B^{(2)}=0$) then $o(A)=k$. Indeed, if $o(A)=i<k$ then we have $0\ne ix\in Z(B,+)$,
        a contradiction (see the comments before Theorem~\ref{CorSummary} or Lemma~\ref{Bprop}(3), as $B^{(2)}=0$ yields $Z(B,+)=\Soc(B)=0$).
        \item Assume $A'=1$ and $v_1\ne 0$, in particular $B^{(2)}\ne 0$. Then the condition \eqref{OrdCond} is equivalent with:
        \begin{equation}
            \text{if }A^i=1\text{ for some }0<i<k\text{ then $p$ does not divide }i.\label{a'=1}
        \end{equation}
        Next, since $\lambda^k=\id$, we get $x=\lambda^k(x)=kv_1+x$. So, $kv_1=0$ and thus $p$ divides $k$ (i.e., $s>0$)
        because, by assumption, $v_1\ne 0$. Further, $A^m=1$ by (2), which yields $A^{p^{s-1}m}=1$, and thus $s=1$ by \eqref{a'=1}
        as $0<p^{s-1}m<k$. Hence, $k=pm$. Finally, we claim that $o(A)=m$. Indeed, if $o(A)<m$ then $0<o(A)p<pm=k$ and so $A^{o(A)p}=1$
        leads to a contradiction with \eqref{a'=1}.
        \item Assume $A'\ne 1$. Then, since $\mc{R}$ is a field, the condition \eqref{OrdCond} is equivalent with $A^i\ne 1$
        or $A'^i\ne 1$ for $0<i<k$, because $A'^i-1=(A'-1)(1+A'+\dotsb+A'^{i-1})$ yields that $A'^i\ne 1$ if and only if
        $1+A'+\dotsb+A'^{i-1}\ne 0$. Hence, we cannot have $i=\lcm(o(A),o(A'))<k$ because $A^i=1$ and $A'^i=1$. Therefore,
        in this case, the condition \eqref{OrdCond} is equivalent with $k=\lcm(o(A),o(A'))$ and thus $p$ does not divide $k$.
    \end{enumerate}
\end{rem}

Therefore, we have proved the last part of the following result which is a refinement of Theorem~\ref{CorSummary}(2) in case
$B^{(3)}=0$ and $V$ is abelian.

\begin{thm}\label{coro1*}
    Assume that:
    \begin{itemize}
        \item $V$ is a finite elementary abelian $p$-group of order $p^n$ (for some prime $p$ and some $n\ge 1$),
        \item $C$ is a finite cyclic group of order $k$ (for some $k>1$) with a generator denoted by $x$,
        \item  $A,A'\in\Aut(V)$ are commuting automorphisms of orders dividing
        $\gcd(k,p^n-1)$ such that $A\ne 1$, $u_0\in V$ (put $v_1=u_0-Au_0\in V$) and:
        \begin{itemize}
            \item if $A'=1$ and $v_1=0$ then $k=o(A)$,
            \item if $A'=1$ and $v_1\ne 0$ then $k=o(A)p$, 
            \item if $A'\ne 1$ then $k=\lcm(o(A),o(A'))$,
        \end{itemize}
        \item $V$ is a simple $\mc{R}$-module, where $\mc{R}$ is the subring of $\End(V)$ generated by $A$ and $A'$
        (and thus $\mc{R}$ is a field with $p^n$ elements).
    \end{itemize}
    Let $(B,+)=V\rtimes C$ (we shall denote the elements $a\in B$ as sums $a=v+ix$ with $v\in V$ and $0\le i<k$, and write $|a|=i$),
    where the action of $C$ on $V$ is determined by the automorphism $A$ (i.e., $x+v-x=Av$ for $v\in V$). Further, let
    $a\circ b=a+\lambda^{|a|}(b)$ for $a,b\in B$, where $\lambda\colon B\to B$ is defined as $\lambda(v+ix)=A'v+u_0-A^iu_0+ix$.
    Then $(B,+,\circ)$ is a skew left brace and the solution $(B,r_B)$ determined by $B$ restricts to a (finite non-degenerate irretractable)
    simple solution $(X,r)$ on the set $X=\{v+x-v:v\in V\}$, the additive conjugacy class of $x$. Further, $X=V+x$ and
    \[r(v+x,w+x)=(A'w+v_1+x,-A^{-1}w+(AA')^{-1}(v-v_1)+w+x).\]

    Conversely, finite skew left braces $B$ that define (finite non-degenerate irretractable) simple solutions $(X,r)$
    of type $(2)$ in Theorem~\ref{CorSummary}, with $B^{(3)}=0$ and $V$ abelian, satisfy all the conditions stated above.
    \begin{proof}
        We only need to prove that the listed conditions are sufficient for $(X,r)$ to be a finite non-degenerate simple solution.
        To do this it is enough to check that the conditions stated in Theorem~\ref{CorSummary}(2) hold.
        
        First, we check that $\lambda$ is an automorphism of $(B,+)$. Let $a=v+ix\in B$ and $b=w+jx\in B$,
        where $v,w\in V$ and $0\le i,j<k$. If $0\le l<k$ satisfies $i+j\equiv l\pmod k$ then
        \begin{align*}
            \lambda(a+b) & =\lambda(v+A^iw+(i+j)x)
             =\lambda(v+A^iw+lx)
             =A'v+A'A^iw+u_0-A^lu_0+lx
        \end{align*}
        and
        \begin{align*}
            \lambda(a)+\lambda(b) & =A'v+u_0-A^iu_0+ix+A'w+u_0-A^ju_0+jx\\
            & =A'v+u_0-A^iu_0+A^iA'w+A^iu_0-A^{i+j}u_0+(i+j)x\\
            & =A'v+A^iA'w+u_0-A^lu_0+lx
        \end{align*}
        because $A^k=1$ yields $A^{i+j}=A^l$. Therefore, as $AA'=A'A$, we conclude that
        $\lambda(a+b)=\lambda(a)+\lambda(b)$, and thus indeed $\lambda\in\Aut(B,+)$ (bijectivity of $\lambda$ is obvious).

        Moreover, $(B,+,\circ)$ is a skew left brace, because if $a,b,c\in B$ (with $a$ and $b$ written
        in the form as above) then we get
        \begin{align}
            \begin{aligned}\label{brace}
                a\circ(b+c) & =a+\lambda^{|a|}(b+c)=a+\lambda^{|a|}(b)+\lambda^{|a|}(c)\\
                 & =a+\lambda^{|a|}(b)-a+a+\lambda^{|a|}(c)=a\circ b-a+a\circ c
            \end{aligned}
        \end{align}
        and 
        \begin{equation}\label{brace2}
            \lambda_{a+\lambda_a(b)}=\lambda_a\lambda_b
        \end{equation}
        because $|a+\lambda_a(b)|\equiv|a|+|b|\pmod k$ (as $a+\lambda_a(b)\in V+lx$) and $\lambda^k=\id$. Note that the
        equality $\lambda^k=\id$ is a consequence of our assumptions $A'^k=1$ and $\lambda^k(x)=x$. Indeed, the latter follows
        because if $w_n=A'^{n-1}v_1+\dotsb+A'v_1+v_1\in V$ for $n\ge 0$ (here and later we understand that expressions such as $w_n$
        are equal to zero in case $n=0$; this convention will allow us to omit splitting such formulas into subcases when $n=0$ or $n>0$)
        then
        \begin{align}\label{lambda-n}
            \begin{aligned}
                \lambda^n(x) & =\lambda^{n-1}(v_1+x)=\lambda^{n-1}(v_1)+\lambda^{n-1}(x)
                 =A'^{n-1}v_1+\dotsb+A'v_1+v_1+x=w_n+x.
            \end{aligned}
        \end{align}
        In particular, if $A'=1$ and $v_1\ne 0$ then, by assumptions, $k=o(A)p$. Thus $w_k=kv_1=0$ and \eqref{lambda-n} yields $\lambda^k(x)=x$.
        Whereas if $A'\ne 1$ then $A'^k=1$ implies $1+A'+\dotsb+A'^{k-1}=0$. So, again $w_k=0$ and $\lambda^k(x)=x$ by \eqref{lambda-n}.
    
        Next, if $a,b\in B$ are written as before then \eqref{lambda-n} yields
        \begin{align}\label{a*b}
            \begin{aligned}
                a*b & =\lambda_a(b)-b
                 =\lambda^i(w+jx)-(w+jx)
                 =\lambda^i(w)+j\lambda^i(x)-jx-w\\
                & =A'^iw+j(w_i+x)-jx-w
                 =A'^iw+w_i+Aw_i+\dotsb+A^{j-1}w_i-w\in V.
            \end{aligned}
        \end{align} 
        Hence, $B^{(2)}\s V$ and thus, because $B^{(2)}$ is an $\mc{R}$-submodule of $V$, we get that  $B^{(2)}=V$ if $B^{(2)}\ne 0$.
        Note that  $B^{(2)}=0$ implies that $A'=1$ (just take $i=1$ and $j=0$ in \eqref{a*b}) and consequently $v_1=w_1=0$ (take $i=1$
        and $j=1$ in \eqref{a*b}). So, $B^{(2)}=0$ is equivalent with $A'=1$ and $v_1=0$. In this case it is easily seen that $V=[B,B]$. 
        
        Next, we show that $V$ is the smallest non-zero ideal of the skew left brace $B$. To do so it is sufficient to prove that for any
        non-zero ideal $I$ of $B$ we have $I\cap V \ne 0$. Indeed, since $V$ is a simple $\mc{R}$-module and  $I\cap V$ is an $\mc{R}$-submodule
        of $V$, then $I\cap V\ne 0$  yields $I\cap V=V$ and thus $V\s I$. So, assume that $I$ is a non-zero ideal of $B$. Suppose that $I\cap V=0$.
        Choose $0\ne a\in I$ and write $a=v+ix$ for some $v\in V$ and $0\le i<k$. If $v\ne 0$ then $b=x+a-x=Av+ix\in I$ satisfies $b\ne a$
        because, by the assumptions $A-1$ is invertible, and thus $Av\ne v$. Hence, we have $0\ne b-a=Av-v\in I\cap V$, a contradiction.
        Thus, $v=0$ and $a=ix$ with $i>0$ as $a\ne 0$. Since $I$ is a normal subgroup of $(B,+)$, we also get $w-A^iw+a=w+a-w\in I$ and,
        in consequence, $w-A^iw\in I$ for all $w\in V$. Hence, $(A^i-1)V\s I\cap V=0$. So, $A^i=1$ and thus $o(A)$ divides $i$;
        let us write $i=o(A)j$ for some $j$. So, if $B^{(2)}=0$ then $A'=1$. Thus, by the assumption, $o(A) = k$ and therefore $a=v$,
        a contradiction. From now on, we assume that $B^{(2)}\ne 0$ and thus $0\ne B^{(2)}=V$ and $1<o(A)<k$. We deal separately with
        two mutually exclusive cases: $A'\ne 1$ and $A'=1$.
        
        Case $A'\ne 1$. As $0<i<k$, it follows from the assumptions that $o(A')$ does not divide $i$. Therefore, $A'^i\ne 1$. Hence,
        $\lambda^i\ne\id$ and thus $\lambda^i(b)\ne b$ for some $b\in B$. Consequently,
        $0\ne\lambda^i(b)-b=\lambda_{ix}(b)-b=\lambda_a(b)-b=a*b\in I\cap V$, a contradiction. Thus, we must have $I\cap V\ne 0$. 
        
        Case $A'=1$. In this situation we know that $v_1\ne 0$ and, from the assumptions, $k=o(A)p$. Since $i=o(A)j$ satisfies $0<i<k$,
        we get $0<j<p$. In particular, $j$ is not divisible by $p$. Because also $o(A)$ is not divisible by $p$, it follows
        that $i$ is not divisible by $p$. Hence, $\lambda^i(x)=iv_1+x\ne x$ and thus $0\ne iv_1=\lambda^i(x)-x=\lambda_a(x)-x=a*x\in I\cap V$,
        a contradiction. Hence, we indeed have shown that $V$ is the smallest non-zero ideal of $B$.

        Note also that $\lambda_v=\lambda^0=\id$ for any $v\in V$, and thus, if $B^{(2)}\ne 0$, then $B^{(3)}=B^{(2)}=V*B=0$ because
        $v*b=\lambda_v(b)-b=0$ for each $b\in B$. Moreover, it also follows that the action of the ideal $V$ on the set $X=\{v+x-v:v\in V\}$
        is transitive. Further, since $1-A$ is invertible, we get that \[X=\{v-Av+x:v\in V\}=(1-A)V+x=V+x\] and, in consequence,  $V=\{y-x:y\in X\}$.
        Hence, in particular, $B$ is additively generated by $X$. Finally, it is clear that that $B/V$ is a trivial skew left brace of cyclic type.
        This finishes the proof of the sufficiency of the listed condition.

        Therefore, it remains to show that the solution $(X,r)$ is of the form as stated in the result. So, let $a=v+x\in X$ and $b=w+x\in X$
        for some $v,w\in V$. Then $\lambda_a(b)=\lambda(w+x)=A'w+v_1+x$. Moreover, since $\lambda_a=\lambda_b=\lambda$ and $\lambda_{-b}=\lambda^{-1}$,
        we get $(-b)\circ\lambda_a(b)=-b+\lambda_{-b}(\lambda_a(b))=-b+\lambda^{-1}(\lambda(b))=0$, which yields
        $\lambda_a(b)^{-1}=-b$. Hence, in consequence,
        \begin{align}\label{rho}
            \begin{aligned}
            \rho_b(a) & =\lambda_a(b)^{-1}\circ a\circ b=(-b)\circ(a+\lambda(b))=-b+\lambda^{-1}(a+\lambda(b))\\
            & =-b+\lambda^{-1}(a)+b=-x-w+A'^{-1}(v-v_1)+x+w+x\\
            & =A^{-1}(-w+A'^{-1}(v-v_1))+w+x
            \end{aligned}
        \end{align}
        (the last equality follows because $-x+u=A^{-1}u-x$ for $u\in V$). Hence
        \[r(v+x,w+x)=(A'w+v_1+x,-A^{-1}w+(AA')^{-1}(v-v_1)+w+x),\] which completes the proof.
    \end{proof}
\end{thm}

Although Theorem~\ref{coro1*} gives many examples, we provide a concrete construction that,
under certain extra conditions, yields solutions of this type.

\begin{ex}\label{ex-ab}
    Let $V=\mb{F}_p^n$ for some prime $p$ and some $n\ge 1$. Moreover, let $C=\free{x}$ be a cyclic group of order $k=2n$. Define
    $A=A'\in\Aut(V)$ as \[A(s_1,\dotsc,s_n)=(s_n,-s_1,s_2,\dotsc,s_{n-1})\] (note that $A=A'=-1$ for $n=1$). Clearly, $o(A)=o(A')=k$.
    Moreover, it is easy to verify that $Z(B,+)=0$. Further, assume that $u_0=0$ (so, $\lambda$ is the conjugation by $x$)
    and $k\mid p^n-1$. In order for this data to yield and example we also need that $V$ is a simple $\mc{R}$-module.
    \begin{enumerate}
        \item If $n=1$ then $(B,+)\cong C_p\rtimes C_2\cong D_{2p}$. Moreover, $X$ is the set of elements of order $2$ in $D_{2p}$
        and clearly $V=B^{(2)}$ is the smallest non-zero ideal of $B$.
        \item If $n=2$ then $V=B^{(2)}$ is the smallest non-zero ideal of $B$ if and only if $p\equiv 3\pmod 4$.
        Indeed, if $I$ is a non-zero ideal of $B$ then $[I,B]_+\ne 0$ as $Z(B,+)=0$. Since $[I,B]_+\s I\cap V$, there exists
        $0\ne v=(s_1,s_2)\in I\cap V$. If $s_i=0$ for some $1\le i\le 2$ then it is easy to check that $V\s I$. If $s_1=s_2\ne 0$
        then $(s_1,-s_1)=Av\in I$ and $(2s_1,0)=v+Av\in I$. Since $2s_1\ne 0$ we are done as in the previous case. Whereas, if $s_1\ne s_2$
        and both are non-zero then  $v\in I$ and $Av=(s_2,-s_1)\in I$ are linearly independent over $\mb{F}_p$ (and thus span $V$,
        which yields $V\s I)$ if and only $s_1^2+s_2^2\ne 0$, or equivalently $p\equiv 3\pmod 4$ (because the equation $x^2=-1$
        has a solution in $\mb{F}_p$ if and only if $p\equiv 1\pmod 4$).
        \item On the other hand, if $n=3$ and $v=(1,1,2)\in V$ then $Av=(2,-1,1)$ and $A^2v=(1,-2,-1)$. Hence, $a-Av+A^2v=0$
        and thus $I=\free{v,Av}_+$ is a non-zero ideal of $B$ properly contained in $V=B^{(2)}$. So, for arbitrary $n$
        this construction does not yield examples of the desired type.
    \end{enumerate}
\end{ex}

\begin{ex}
    Let $V$ be the additive group of the finite field $\mb{F}_{p^n}$, for some prime $p$ and some $n\ge 1$ (obviously $V\cong\mb{F}_p^n$),
    and let $C=\free{x}$ be a cyclic group of order $k=p^n-1$. Define $A=A'\in\Aut(V)$ as the multiplication by a generator $\xi$ of the 
    multiplicative group of $\mb{F}_{p^n}$ (i.e., $Av=\xi v$ for $v\in V$) and $u_0=0$. Clearly $o(A)=o(A')=k$. Furthermore, $\mc{R}v=V$
    for any $0\ne v\in V$ and thus $V$ is a simple $\mc{R}$-module. Concluding, all assumptions of Theorem~\ref{coro1*} are satisfied.
\end{ex}

\subsection{The non-abelian case}
Now, let us assume that $V$ is non-abelian. As we already know (see Lemma~\ref{charsimp}), in this case $V=S_1+\dotsb+S_n$,
a direct sum of isomorphic non-abelian simple groups, say isomorphic with $S$.

Further, Remark~\ref{elform} guarantees that $B=V+C$, where $C=\free{x}_+$, a cyclic group of order $m$, for some $m$.
Since $B/V\cong\free{x+V}_+$, we clearly get that $k=|B/V|$ divides $m$. Since $V\ne B$, we get $k>1$, as before.
However, in contrast to the case in which $V$ is abelian, it may happen that $kx\ne 0$ and thus $V\cap C\ne 0$.

Let $D=V\rtimes C$ be the semi-direct product of the additive groups $V$ and $C$, where the action of $C$ on $V$ is determined by
the automorphism $A$ (introduced just before Lemma~\ref{Asim}), i.e., 
$(v,ix)+(w,jx)=(v+A^iw,(i+j)x)$ for $(v,ix)\in D$ and $(w,jx)\in D$.
We claim that $(B,+)\cong D/Z(D)$ Indeed, taking into account
that \[Z(D)=\{(v,ix)\in D:Av=v\text{ and }A^iw=-v+w+v\text{ for each }w\in V\},\] it is easy to verify that the map $\phi\colon D\to B$,
defined as $\phi(v,ix)=v+ix$, is a surjective morphism of groups with $\Ker(\phi)=Z(D)$. Hence, $\phi$ induces the desired isomorphism.

Next, observe that \[A'(kx)=\lambda(kx)=k\lambda(x)=k(u_0+x-u_0)=u_0+kx-u_0=u_0-A^ku_0+kx=v_k+x.\]

Finally, we claim that $V=\free{v-Av:v\in V}_+$, that is $V$ is additively generated by elements of the form $v-Av$ with $v\in V$.
Moreover, as we have $v-Av=y-x$ for $y=v+x-v\in X$, our claim may be equivalently expressed as that $V$ is additively generated
by elements of the form $y-x$ with $y\in X$ or, by elements of the form $y-z=(y-x)-(z-x)$ with $y,z\in X$. To prove the claim put
$I=\free{y-z:y,z\in X}_+$. Since the set $X$ is a conjugacy class that is both $A$-invariant and $A'$-invariant
(and thus $\mc{A}$-invariant), it follows that $I$ is an $\mc{A}$-invariant normal subgroup of $(V,+)$. Since $|X|>1$, we get
$I\ne 0$ and thus $I=V$ by Lemma~\ref{Asim}, as claimed.

Therefore, we have proved the necessity part of the following refinement of Theorem~\ref{CorSummary}(2) in case $B^{(3)}=0$ and $V$ is non-abelian.

\begin{thm}\label{coro2*}
    Assume that:
    \begin{itemize}
        \item $V$ is a finite non-abelian characteristically simple group, i.e., $V=S_1+\dotsb+S_n$ is a direct sum of isomorphic
        finite non-abelian simple groups (say, isomorphic with $S$),
        \item $C$ is a finite cyclic group of order $m$ (for some $m>1$) with a generator denoted by $c$,
        \item $D=V\rtimes C$, where the action of $C$ on $V$ is determined by an automorphism $1\ne A\in\Aut(V)$ (i.e.,
        $(v,ic)+(w,jc)=(v+A^iw,(i+j)c)$ for $(v,ic)\in D$ and $(w,jc)\in D$).
    \end{itemize}    
    Let $(B,+)=D/Z(D)$.  (Since $Z(V)=0$, it follows that $V$ embeds in $B$ as a normal subgroup, so we may write $B=V+\free{x}_+$,
    where $x$ is the image of $c$ in $B$. Further, if $k=|B/V|$ then each element $a\in B$ may be uniquely written as $a=v+ix$ for
    some $v\in V$ and $0\le i<k$; we shall write $|a|=i$.) Assume, furthermore, that $k>1$ and:    
    \begin{itemize}    
        \item $u_0\in V$ (put $v_n=u_0-A^nu_0\in V$ for $n\ge 0$; note that $v_0=0$) and $A'\in\Aut(V)$ is an automorphism of order $o(A')$
        dividing $k$ such that $[A,A']v=-v_1+v+v_1$ for $v\in V$, $A'(kx)=v_k+kx$ and $A'^{k-1}v_1+\dotsb+A'v_1+v_1=0$,
        \item $V$ is additively generated by elements of the form $v-Av$ for $v\in V$, and it has no non-trivial normal $\mc{A}$-invariant
        subgroups, where $\mc{A}$ is the subgroup of $\Aut(V)$ generated by $A$ and $A'$.    
    \end{itemize}
    Let
    $a\circ b=a+\lambda^{|a|}(b)$ for $a,b\in B$, where $\lambda\colon B\to B$ is defined as $\lambda(v+ix)=A'v+v_i+ix=A'v+u_0-A^iu_0+ix$.
    Then $(B,+,\circ)$ is a skew left brace satisfying the conditions stated in Theorem~\ref{CorSummary}(2) and the solution $(B,r_B)$
    determined by $B$ restricts to a (finite non-degenerate irretractable) simple solution $(X,r)$ on the set $X=\{v+x-v:v\in V\}$,
    the additive conjugacy class of $x$. Further, $r$ is given as \[r(v+x-v,w+x-w)=(w'+x-w',v'+x-v'),\] where
    \[w'=A'w+u_0\quad\text{and}\quad v'=w-A^{-1}w+(A'A)^{-1}(v-u_0).\]
    (Note that the elements $v',w'\in V$ are not uniquely determined by the elements $v+x-v\in X$ and $w+x-w\in X$.
    But if $v+x-v=\tilde{v}+x-\tilde{v}$ and $w+x-w=\tilde{w}+x-\tilde{w}$ for some $\tilde{v},\tilde{w}\in V$ then $v'+x-v'=\tilde{v}'+x-\tilde{v}'$
    and $w'+x-w'=\tilde{w}'+x-\tilde{w}'$, so $r$ is indeed well-defined.)

    Conversely, finite skew left braces $B$ that define (finite non-degenerate irretractable) simple solutions $(X,r)$
    of type $(2)$ in Theorem~\ref{CorSummary}, with $B^{(3)}=0$ and $V$ non-abelian, satisfy all the conditions stated above.
    \begin{proof}
        We only need to prove that the listed conditions are sufficient for a finite non-degenerate simple solution $(X,r)$.
        To do this it is enough to check that the conditions stated in Theorem~\ref{CorSummary}(2) hold.
        
        First, we check that $\lambda$ is an automorphism of $(B,+)$. Let $a=v+ix\in B$ and $b=w+jx\in B$, where $v,w\in V$ and $0\le i,j<k$.
        Write $i+j=qk+l$ for some $q\ge 0$ and $0\le l<k$ (note that $i,j<k$ implies that actually $q=0$ or $q=1$).
        To prove that $\lambda(a+b)=\lambda(a)+\lambda(b)$ we shall use the following formula
        \begin{equation}\label{mor1}
            v_n+A^nA'w=A'A^nw+v_n,
        \end{equation}
        which we claim to hold for any $w\in V$ and any $n\ge 0$. Clearly \eqref{mor1} is true for $n=0$. Moreover, \eqref{mor1} for $n=1$
        is just our assumption that $[A,A']w=-v_1+w+v_1$. Now, suppose that \eqref{mor1} holds for some $n\ge 0$. Then, using
        \begin{equation}\label{vprop}
            v_{s+t}=v_s+A^sv_t,
        \end{equation}
        which holds for any $s,t\ge 0$, we get, due to the inductive assumption (applied to $Aw$ in place of $w$), that
        \begin{align*}
            v_{n+1}+A^{n+1}A'w & =v_n+A^nv_1+A^{n+1}A'w
            =v_n+A^n(v_1+AA'w)
             =v_n+A^n(A'Aw+v_1)\\
            & =v_n+A^nA'(Aw)+A^nv_1
             =A'A^n(Aw)+v_n+A^nv_1=A'A^{n+1}w+v_{n+1}.
        \end{align*}
        Now, taking into account \eqref{mor1}, \eqref{vprop} and
        \begin{align*}
            A'(qkx) & =qA'(kx)=q(v_k+kx)=q(u_0-A^ku_0+kx)\\
            & =q(u_0+kx-u_0)=u_0+qkx-u_0=u_0-A^{qk}u_0+qkx=v_{qk}+qkx,
        \end{align*}
        we obtain
        \begin{align*}
            \lambda(a+b) & =\lambda(v+A^iw+qkx+lx)=A'v+A'A^iw+A'(qkx)+v_l+lx\\
             & =A'v+A'A^iw+v_{qk}+qkx+v_l+lx=A'v+A'A^iw+v_{qk}+A^{qk}v_l+(qk+l)x\\
             & =A'v+A'A^iw+v_{i+j}+(i+j)x=A'v+A'A^i w+v_i+A^iv_j+(i+j)x\\
             & =A'v+v_i+A^iA'w+A^iv_j+ix+jx=A'v+v_i+ix+A'w+v_j+jx=\lambda(a)+\lambda(b).
        \end{align*}
        Therefore, $\lambda\in\Aut(B,+)$, since $\lambda$ is obviously bijective.

        Further, that $(B,+,\circ)$ is a skew left brace can be confirmed by exactly the same calculation as in \eqref{brace}
        and \eqref{brace2} in the proof of Theorem~\ref{coro1*} (here, to obtain that $\lambda^k=\id$, we use the assumptions
        that $A'^k=1$ and $A'^{k-1}v_1+\dotsb+A'v_1+v_1=0$).

        Next, as in the proof of Theorem~\ref{coro1*} (see the paragraph containing \eqref{a*b}), we show that $B^{(2)}\s V$.
        Since $B^{(2)}$ is a non-zero normal $\mc{A}$-invariant subgroup of $V$, we get $B^{(2)}=V$ by our assumptions.

        Now, we claim that $V$ is the smallest non-zero ideal of $B$. So, assume $I$ is a non-zero ideal of $B$.
        We need to show that $V\s I$. To prove the claim we  consider two mutually exclusive  cases.
        
        Case I. Assume first that $I\cap Z(B,+)=0$. Choose $0\ne a=v+ix\in I$. Since $a\notin Z(B,+)$, by Lemma~\ref{Bprop}(1) we have
        that either $Av\ne v$ or $A^iw\ne-v+w+v$ for some $w\in V$. In the former situation $0\ne Av-v=(Av+ix)-(v+ix)=(x+a-x)-a\in I\cap V$.
        Whereas, in the latter situation, $0\ne w+v-A^iw-v=(w+v-A^iw+ix)-(v+ix)=(w+a-w)-a\in I\cap V$. Hence, in both situations,
        $I\cap V\ne 0$ and thus $I\cap V=V$. This clearly yields $V\s I$, as desired.
        
        Case II. Now, assume that $I\cap Z(B,+)\ne 0$. Choose $0\ne a=v+ix\in I\cap Z(B,+)$. Again, because of Lemma~\ref{Bprop}(3), we get that
        $\Soc(B)=0$. Therefore, $a\notin\Soc(B)$ and consequently $\lambda_a\ne\id$. Hence, $\lambda_a(b)\ne b$ for some $b\in B$ and thus
        $0\ne\lambda^i(b)-b=\lambda_a(b)-b=a*b\in I\cap V$. Hence, again $I$ contains $V$, as desired.
        
        Hence, we indeed have shown that $V$ is the smallest non-zero ideal of $B$.
        
        Note also that $\lambda_v=\lambda^0=\id$ for any $v\in V$, and thus $B^{(3)}=B^{(2)}*B=V*B=0$ because $v*b=\lambda_v(b)-b=0$ for
        each $b\in B$. Moreover, it also follows that the action of the ideal $V$ on the set $X=\{v+x-v:v\in V\}$ is transitive. Further, as
        $V$ is, by our assumption, additively generated by elements of the form $v-Av=(v+x-v)-x$ with $v\in V$, we conclude that $X$ generates $B$
        additively. Finally, it is clear that that $B/V$ is a trivial skew left brace of cyclic type. This finishes the proof of the sufficiency
        of the listed conditions.
        
        Therefore, it remains to show that the solution $(X,r)$ is of the form as stated in the result. 
        So, let $a=v+x-v\in X$ and $b=w+x-w\in X$ with $v,w\in V$. Then, as $\lambda_a=\lambda$, we get
        \begin{align*}
            \lambda_a(b) & =\lambda(w+x-w)=\lambda(w)+\lambda(x)-\lambda(w)=A'w+u_0+x-u_0-A'w\\
            &=(A'w+u_0)+x-(A'w+u_0)=w'+x-w'.
        \end{align*}
        Next, analogously as in \eqref{rho}, we show that $\rho_b(a)=-b+\lambda^{-1}(a)+b$, and thus, using the equality
        $\lambda^{-1}(x)=-A'^{-1}u_0+x+A'^{-1}u_0$, we obtain
        \begin{align*}
            \rho_b(a) & =-b+\lambda^{-1}(a)+b\\
            & =w-x-w+A'^{-1}v-A'^{-1}u_0+x+A'^{-1}u_0-A'^{-1}v+w+x-w\\
            & =w-A^{-1}w-x+A'^{-1}v-A'^{-1}u_0+x+A'^{-1}u_0-A'^{-1}v+x+A^{-1}w-w\\
            & =w-A^{-1}w+(A'A)^{-1}v-x-A'^{-1}u_0+x+A'^{-1}u_0+x-(A'A)^{-1}v+A^{-1}w-w\\
            & =w-A^{-1}w+(A'A)^{-1}(v-u_0)+x+(A'A)^{-1}(u_0-v)+A^{-1}w-w\\
            & =(w-A^{-1}w+(A'A)^{-1}(v-u_0))+x-(w-A^{-1}w+(A'A)^{-1}(v-u_0))=v'+x-v',
        \end{align*}
        which completes the proof.
    \end{proof}
\end{thm}

\begin{rem}
    Observe that $k$ must divide $m$, the order of $C=\free{c}$. Indeed, suppose that $k$ does not divide $m$. Then $m=qk+l$ for some $q\ge 0$
    and $0<l<k$. Since $mx=0$ (as $mc=0$) and $kx\in V$, we get $lx=-qkx\in V$, and thus $k=|B/V|\le l$, a contradiction.
\end{rem}

We know that  if $B^{(2)}=0$ then $V=[B,B]$, and $V=B^{(2)}$ otherwise. Thus, in both cases,
$V$ is uniquely determined by the isomorphism class of the skew left brace $B$. Hence, the number $k=|B/V|$ is uniquely
determined and, by the Jordan--H\"older theorem, also the number $n$ of simple factors of $V$ is uniquely determined.

Skew left braces $B$ as in Theorem~\ref{coro1*} and Theorem~\ref{coro2*} are determined by a simple group $S$,
a positive integer $n$ (this data determines the additive group $V$, and $V$ is abelian if and only if the simple group
$S$ is cyclic), an integer $k>1$, two elements $A$ (determining conjugation by $x$ on $V$) and $A'$ (determining restriction
of $\lambda$ to $V$) in $\Aut(V)$, an element $x\in B$ such that $x+V\in B/V$ is an element of order $k$, and an element
$u_0\in V$ (also needed to determine $\lambda$).

\begin{prop}
    Two skew left braces $B$ and $\tilde{B}$ as in Theorem~\ref{coro1*} or in Theorem~\ref{coro2*} are isomorphic if and only if:
    \begin{itemize}
        \item the corresponding simple groups $S$ and $\tilde{S}$ are isomorphic,
        \item the corresponding integers $n$ and $\tilde{n}$ are equal (note that this condition together with the previous one imply that
        the corresponding groups $V$ and $\tilde{V}$ are isomorphic),
        \item the corresponding integers $k$ and $\tilde{k}$ are equal (in particular $B=V+\free{x}_+$ and $\tilde{B}=\tilde{V}+\free{\tilde{x}}_+$
        for some $x\in B$ and $\tilde{x}\in\tilde{B}$ such that $x+V\in V/B$ and $\tilde{x}+\tilde{V}\in\tilde{B}/\tilde{V}$ are elements of order $k$),
        \item there exists an isomorphism $f\colon V\to\tilde{V}$ of groups such that:
        \begin{itemize}
            \item $f(kx)=k\tilde{x}$,
            \item $f(u_0)+\tilde{x}-f(u_0)=\tilde{u}_0+\tilde{x}-\tilde{u}_0$ (i.e., $-\tilde{u}_0+f(u_0)$ centralizes $\tilde{x}$),
            \item for the corresponding automorphisms $A,A'\in\Aut(V)$ and $\tilde{A},\tilde{A}'\in\Aut(\tilde{V})$
            (conjugations by $x$ and $\tilde{x}$, and restrictions of $\lambda$ and $\tilde{\lambda}$)
            we have $\tilde{A}=fAf^{-1}$ and $\tilde{A}'=fA'f^{-1}$.
        \end{itemize}
    \end{itemize}
    \begin{proof}
        Suppose $F\colon B\to\tilde{B}$ is an isomorphism of skew left braces satisfying conditions as in Theorem~\ref{CorSummary}(2).
        By the previous paragraph we get $F(V)=\tilde{V}$. Let $f\colon V\to\tilde{V}$ denote the isomorphism induced by $F$ (its restriction).
        Take $x\in X$ and put $\tilde{x}=F(x)\in\tilde{X}$. Then $f(Av)=F(x+v-x)=\tilde{x}+f(v)-\tilde{x}=\tilde{A}f(v)$ for each $v\in V$.
        Hence, $\tilde{A}=fAf^{-1}$. Further, since $F\lambda=\tilde{\lambda}F$, we get $f(A'v)=F(\lambda(v))=\tilde{\lambda}(F(v))=\tilde{A}'f(v)$
        for each $v\in V$. Therefore, $\tilde{A}'=fA'f^{-1}$. Finally,
        \[f(u_0)+\tilde{x}-f(u_0)=F(u_0+x-u_0)=F(\lambda(x))=\tilde{\lambda}(F(x))=\tilde{\lambda}(\tilde{x})=\tilde{u}_0+\tilde{x}-\tilde{u}_0.\]
        
        Conversely, if the above conditions hold then the map $F\colon B\to\tilde{B}$, defined as $F(v+ix)=f(v)+i\tilde{x}$ for $v\in V$ and
        $0\le i<k$, is an isomorphism of skew left braces. Indeed, first note that $F$ is well-defined because each element of $B$ has
        a unique presentation of the form $v+ix$ with $v\in V$ and $0\le i<k$. Next, if $a=v+ix\in B$ and $b=w+jx\in B$ for some
        $v,w\in V$ and $0\le i,j<k$ then write $i+j=qk+l$ for some $q\ge 0$ and $0\le l<k$. Then
        \begin{align*}
            F(a+b) & =F(v+A^iw+qkx+lx)=f(v)+f(A^iw)+qf(kx)+l\tilde{x}\\
            & =f(v)+\tilde{A}^if(w)+qk\tilde{x}+l\tilde{x}=f(v)+i\tilde{x}+f(w)+j\tilde{x}=F(a)+F(b).
        \end{align*}
        Further, using the additivity of $F$, we get
        \begin{align*}
            F(\lambda(a)) & =F(A'v+i(u_0+x-u_0))=f(A'v)+i(f(u_0)+\tilde{x}-f(u_0))\\
            & =\tilde{A}'f(v)+i(\tilde{u}_0+\tilde{x}-\tilde{u}_0)=\tilde{\lambda}(f(v)+i\tilde{x})=\tilde{\lambda}(F(a)).
        \end{align*}
        Therefore, $F\lambda=\tilde{\lambda}F$ and thus
        \begin{align*}
            F(a\circ b) & =F(a+\lambda^{|a|}(b))=F(a)+F(\lambda^{|a|}(b))\\
            & =F(a)+\tilde{\lambda}^{|a|}(F(b))=F(a)+\tilde{\lambda}^{|F(a)|}(F(b))=F(a)\circ F(b)
        \end{align*}
        because $|F(a)|=|a|$. Since the bijectivity of $F$ is clear, $F$ is indeed an isomorphism of skew left braces.
    \end{proof}
\end{prop}

\begin{rem}
    The automorphism $A$, that determines the conjugation by $x$ on $V=S_1+\dotsb+S_n$, is determined by a matrix in $\GL_n(\mb{F}_p)$
    is case $V$ is abelian. In case $V$ is non-abelian, $A$ can be written in a particularly simple form
    provided $A$ acts transitively on the set of simple components $\{S_1,\dotsc,S_n\}$. In this case, without loss of
    generality, we may assume that $A(S_i)=S_{i+1}$ for $1\le i<n$ and $A(S_n)=S_1$. Let $S=S_1$. Since $A\in\Aut(V)$
    and $S_1,\dotsc,S_n$ commute elementwise (i.e., $s_i+s_j=s_j+s_i$ for $s_i\in S_i$ and $s_j\in S_j$ with $i\ne j$),
    we get the group isomorphism $f\colon S^n\to V$ defined as \[f(s_1,\dotsc,s_n)=s_1+As_2+\dotsb+A^{n-1}s_n.\]
    If $T=\{ix:0\le i<k\}$ then this isomorphism can be naturally extended to a bijection $F\colon S^n\times T\to B$
    by defining $F(s,ix)=f(s)+ix$. Hence, we may assume that $V=S^n$ and then $A$ can be simply written as
    \[A(s_1,\dotsc,s_n)=(\theta(s_n),s_1,\dotsc,s_{n-1})\] for some $\theta\in\Aut(S)$ (actually, $\theta$ is just the 
    restriction of $A^n$ to $S$). In other words, the automorphism $A$ is determined by one automorphism of the simple
    group $S$. Therefore, we recover the result of Guralnick recalled in \cite[Theorem 3.7]{MR1994219} and the description
    of simple rack solutions (i.e., solutions with all $\lambda_x=\id$) obtained by Andruskiewitsch and Gra\~na given in that paper.

    In the general case (for $V$ non-abelian), we proceed as follows. Let $S_{ij}=A^{j-1}(S_i)$ for $1\le i\le n$ and $j\ge 1$. Clearly,
    each $S_{ij}$ is also a minimal normal subgroup of $V$. Hence, $S_{ij}\cap U=0$ or $S_{ij}\s U$ for any normal subgroup $U$ of $V$.
    Since $V$ is finite, there exists $k_1$ such that $S_{1,k_1+1}\s S_{11}+\dotsb+S_{1k_1}$. Take a minimal such $k_1$. Then
    $T_1=S_{11}+\dotsb+S_{1k_1}$ is a direct sum and $A(S_{1j})=S_{1,j+1}$ for $1\le j<k_1$. We claim that $A(S_{1k_1})=S_{11}$.
    To prove this it is enough to show that $A(S_{1k_1})\cap S_{11}\ne 0$ because then, by the minimality of $A(S_{1k_1})$ and $S_{11}$,
    we get $A(S_{1k_1})=A(S_{1k_1})\cap S_{11}=S_{11}$, as desired. So, suppose $A(S_{1k_1})\cap S_{11}=0$. Then $[A(S_{1k_1}),S_{11}]=0$.
    Moreover, if $1\le j<k_1$ then $[S_{1j},S_{1k_1}]=0$ implies \[[A(S_{1k_1}),S_{1,j+1}]=[A(S_{1k_1}),A(S_{1j})]=A([S_{1k_1},S_{1j}])=0.\]
    Thus $[A(S_{1k_1}),T_1]=0$ and since $A(S_{1k_1})\s T_1$ we conclude that $A(S_{1k_1})\s Z(T_1)=0$, a contradiction.
    If $T_1\ne V$ then $S_i\nsubseteq T_1$ and thus $S_i\cap T_1=0$ for some $2\le i\le n$. Without loss of generality we may assume that $i=2$.
    Let $k_2$ be minimal such that $S_{2,k_2+1}\s T_1+S_{21}+\dotsb+S_{2k_2}$. Again, the sum $T_2=T_1+S_{21}+\dotsb+S_{2k_2}$ is direct and 
    $A(S_{2j})=S_{2,j+1}$ for $1\le j<k_2$. As before, we show that $[A(S_{2k_2}),S_{1j}]=0$ for $1<j\le k_1$ and $[A(S_{2k_2}),S_{2j}]=0$
    for $1<j\le k_2$. Furthermore, since $A(S_{2k_2})\ne A(S_{1k_1})=S_{11}$, we get $A(S_{2k_2})\cap S_{11}=0$ and thus $[A(S_{2k_2}),S_{11}]=0$.
    Therefore, it follows that $A(S_{2k_2})=S_{21}$ because otherwise $A(S_{2k_2})\cap S_{21}=0$, which yields $[A(S_{2k_2}),S_{21}]=0$.
    Hence, in consequence, $A(S_{2k_2})\s Z(T_2)=0$, a contradiction. Continuing in this manner, reindexing the simple components $S_1,\dotsc,S_n$
    if needed and taking into account that $V$ is finite, we conclude that $V$ can be written as a direct sum $V=T_s=\sum_{i=1}^s\sum_{j=1}^{k_i}S_{ij}$
    such that $A(S_{ij})=S_{i,j+1}$ provided $1\le j<k_i$ and $A(S_{ik_i})=S_{i1}$. So, by the first paragraph, the automorphism $A$ is then
    completely determined by $s$ automorphisms $\theta_i\in\Aut(S)$, one for each $A$-orbit $\{S_{i1},\dotsc,S_{ik_i}\}$.
\end{rem}

\begin{rem}\label{T}
    Assume that $B$ is a skew left brace as constructed in Theorem~\ref{coro2*}.
    Note that in case $V\cap\free{x}_+=0$ (i.e., when $(B,+)=V\rtimes\free{x}_+$ is a semi-direct product)
    the condition $B=\free{X}_+$ implies $V=\free{T}_+$, where $T=\{v-Av:v\in V\}$. Indeed, let $U=\free{T}_+\s V$.
    As $X=\{v+x-v:v\in V\}=\{v-Av+x:v\in V\}=T+x=x+T$ (the last equality follows because $T$ is $A$-invariant),
    we get $B=\free{X}_+=\free{T+x}_+=U+\free{x}_+$. Hence, if $v\in V$ then $v=u+ix$ for some $u\in U$ and some $0\le i<k$,
    and thus $-u+v=ix\in V\cap\free{x}_+=0$. Therefore, $v=u\in U$, and thus $U=V$, as desired.
\end{rem}

\section{Examples }

In the first four examples of this section we include some concrete classes of skew braces that satisfy all the requirements
of the main theorems. In the fifth example we show that our results can be applied to simple skew left braces to generate further
families of simple solutions. And we finish with some remarks on simple solutions of prime cardinality.
    
\begin{ex}\label{ex1*}
    Assume that $S$ is a finite non-abelian simple group and fix an element $0\ne a\in S$ of order $n>1$.
    Put $V=S^n$ and $C=\free{x}$, a cyclic group of order $n$. Next, define $A,A'\in\Aut(V)$ by the following formulas
    \[A(s_1,\dotsc,s_n)=(s_n,s_1,\dotsc,s_{n-1})\quad\text{and}\quad A'(s_1,\dotsc,s_n)=(s_1,\gamma_a(s_2),\dotsc,\gamma_{(n-1)a}(s_n)),\]
    where $\gamma_a\in\Aut(S)$ denotes the conjugation by the element $a$, i.e., $\gamma_a(s)=a+s-a$. Clearly, $o(A)=n$ and also $o(A')=n$
    because $o(\gamma_a)=n$. Moreover, if $D=V\rtimes C$, where the action of $C$ on $V$ is determined by the automorphism $A$
    (as in Theorem~\ref{coro2*}), then taking into account that $Z(S)=0$ one easily verifies that $Z(D)=0$ and thus $(B,+)=D/Z(D)\cong D$.
    In particular, $B/V\cong C$ and $k=|B/V|=n$.
    
    Further, put $u_0=(0,a,2a,\dotsc,(n-1)a)\in V$ and $v_1=u_0-Au_0\in V$. Then $v_1=(a,\dotsc,a)\ne 0$ and $Av_1=A'v_1=v_1$
    (so $A$ has non-zero fixed points). If $v=(s_1,\dotsc,s_n)\in V$ then
    \begin{align*}
        [A,A']v & =AA'A^{-1}(s_1,\gamma_{-a}(s_2),\dotsc,\gamma_{-(n-1)a}(s_n))=AA'(\gamma_{-a}(s_2),\dotsc,\gamma_{-(n-1)a}(s_n),s_1)\\
        & =A(\gamma_{-a}(s_2),\dotsc,\gamma_{-a}(s_n),\gamma_{(n-1)a}(s_1))=(\gamma_{-a}(s_1),\gamma_{-a}(s_2),\dotsc,\gamma_{-a}(s_n))=-v_1+v+v_1
    \end{align*}
    because $\gamma_{(n-1)a}=\gamma_{-a}$ (as we have $na=0$). Moreover, since $kx=nx=0$ and $v_k=v_n=u_0-A^nu_0=0$, we also get $A'(kx)=v_k+kx$.
    
    Now, we claim that $B=\free{X}_+$, where $X=\{v+x-v:v\in V\}$. To prove our claim put $N=\free{X}_+$. Because $B=V+\free{x}_+$ and $x\in N$,
    it suffices to show that $V\s N$. So, let $0\ne s_1\in S$ and define $w=(s_1,0,\dotsc,0)\in V$. Since $N$ is a normal subgroup of $(B,+)$,
    we get $v_1=(s_1,-s_1,0,\dotsc,0)=w-Aw=(w+x-w)-x\in N$. As $Z(S)=0$, it follows that $t=[s_1,s_2]_+\ne 0$ for some $s_2\in S$. Therefore, if
    $v_2=(s_2,0,\dotsc,0)\in V$ then $v=(t,0,\dotsc,0)=([s_1,s_2]_+,0,\dotsc,0)=[v_1,v_2]_+=v_1+(v_2-v_1-v_2)\in N$.
    Further, if $s\in S$ then defining $u=(s,0,\dotsc,0)\in V$, we get $(s+t-s,0,\dotsc,0)=u+v-u\in N$.
    Since $\free{s+t-s:s\in S}_+=S$, we obtain $S\times\{0\}^{n-1}\s N$. Next, applying an analogous argument to elements
    $Av,\dotsc,A^{n-1}v\in N$ (each having just one non-zero component), we get $\{0\}^i\times S\times\{0\}^{n-i-1}\s N$ for each $1\le i<n$.
    Hence, $V=S^n\s N$, as claimed. Therefore, taking into account that $V\cap\free{x}_+=0$, we conclude that $V=\free{v-Av:v\in V}_+$ by Remark~\ref{T}.

    Finally, suppose $I$ is a non-zero $\mc{A}$-invariant normal subgroup of $V$. Choose $0\ne v=(s_1,\dotsc,s_n)\in I$. Since $v\ne 0$,
    we get $s_i\ne 0$ for some $1\le i\le n$. Then, replacing $v$ by $A^{n-i+1}v\in I$, we may assume that $i=1$. Now, repeating the same
    arguments as in the previous paragraph, we reduce to the situation in which $s_2=\dotsb=s_n=0$. Then, we consequently show
    that $S\times\{0\}^{n-1}\s N$ and, again in the same way as before, that $V=S^n\s I$. Hence, $I=V$. Concluding, all assumptions
    of Theorem~\ref{coro2*} are satisfied.
\end{ex}

\begin{ex}\label{ex-sn}
    Consider the symmetric group $\Sym(n)$ of degree $n\ge 5$ with the composition denoted as $+$. The only non-trivial normal subgroup
    of $\Sym(n)$ is the alternating group $\Alt(n)$. Let $X$ be the set of all transpositions in $\Sym(n)$; this is clearly
    a conjugacy class in $\Sym(n)$. Furthermore, if $c=(1,2)\in \Sym(n)$ then $X=\{x-c:x\in X\}+c$ and $\{x-c:x\in X\}$ is a set of
    generators of $\Alt(n)$. To consider $\Sym(n)$ as a skew left brace we need to define a product $\circ$ on $\Sym(n)$ or,
    equivalently, a map $\lambda\colon \Sym(n)\to\Aut(\Sym(n),+)$ satisfying
    \begin{equation}
        \lambda_{a+\lambda_a(b)}=\lambda_a \lambda_b\label{lambdacond}
    \end{equation}
    for all $a,b \in \Sym(n)$ (then $a\circ b=a+\lambda_a(b)$). Any element of $\Sym(n)$ can be written uniquely as $a$ or as $a+c$
    with $a\in \Alt(n)$. Define $\lambda_a=\id$ and $\lambda_{a+c}$ the conjugation by $c$ in $(\Sym(n),+)$, i.e., $\lambda_{a+c}(b)=c+b-c$
    for $b\in \Sym(n)$. With this definition, the condition \eqref{lambdacond} is satisfied. This gives an example of a finite skew
    left brace $B$ of type (2) in Theorem~\ref{CorSummary} with $B^{(2)}$ a non-commutative simple skew left brace. So, we have
    shown the existence of simple solutions of cardinality $|X|=\binom{n}{2}$ for any integer $n\ge 5$. If for example $n=5$ then
    $|X|=10=2\cdot 5 $ and again we have shown the existence of simple solutions of square-free cardinality that are not prime numbers.
    Moreover, considering the trivial skew left brace over $\Sym(n)$ one obtains a simple quandle on the same set $X$, i.e., a simple solution
    of square-free cardinality which is not a prime, see \cite[Example 3.6]{MR1994219}. This is in contrast with the involutive case (see \cite{CO2022})
    where Ced\'o and Okni\'nski have shown that in this case all simple involutive solutions have cardinality which is not
    a square-free number, provided it is not prime. This follows from a much more general result: a finite indecomposable non-degenerate
    involutive solution of the Yang--Baxter equation on a set $X$ of square-free cardinality is a multipermutation solution.
\end{ex}

\begin{ex}
    The same method as in the previous example shows that also the dihedral group $D_{2p}$ of order $2p$, with $p$ a prime number, is a skew
    left brace of type (2) in Theorem~\ref{CorSummary}. Here $X$ again is the set of all elements of order $2$ (see also Example~\ref{ex-ab}).
\end{ex}

\begin{ex}\label{ex-an-pr}
    Assume that $S=\Alt(m)$ is the alternating group of degree $m\ge 5$. Let $V=S^n$ for some $n>1$, and let $C=\free{x}$ be a cyclic
    group of order $2n$. Next, define $A=A'\in\Aut(V)$ by the formula \[A(s_1,\dotsc,s_n)=(s_n,a+s_1-a,s_2,\dotsc,s_{n-1}),\]
    where $a=(1,2)\in \Sym(m)$. Clearly, $o(A)=o(A')=2n$. Now, consider the semi-direct product $D=V\rtimes C$, where the action of $C$
    on $V$ is determined by the automorphism $A$. One easily verifies that $Z(D)=0$, and thus $(B,+)=D/Z(D)\cong D$. In particular,
    $B/V\cong C$ and $k=|B/V|=2n$.

    Further, put $u_0=0$ (hence $\lambda$ is the conjugation by $x$). Then clearly $v_i=u_0-A^iu_0=0$ for any $i\ge 0$.
    Because $A=A'$, we get $[A,A']v=-v_1+v+v_1$ for each $v\in V$ (as both sides are equal to $v$). Next, since $kx=2nx=0$ and
    $v_k=0$, it follows that $A'(kx)=v_k+kx$ (as both sides are equal to zero). Moreover, $A'^{k-1}v_1+\dotsb+A'v_1+v_1=0$.

    Finally, by using a similar method as in Example~\ref{ex1*}, we show that $V=\free{v-Av:v\in V}_+$ and that $V$
    does not have non-trivial $\mc{A}$-invariant (equivalently, $A$-invariant) normal subgroups.
    Concluding, all assumptions of Theorem~\ref{coro2*} are satisfied.
\end{ex}

Recall that any finite group $(G,\circ)$ can be considered as a skew left brace $(G,\circ,\circ)$, the trivial skew left brace
on the group $G$, in particular $G^{(2)}=0$. Note that in this skew left brace $\lambda_a=\id$ and $\rho_a=\sigma_a$ is the conjugation
by $a^{-1}$. Hence, the associated solution is $r_G(a,b)=(b, b^{-1}\circ a \circ b)$ is a left derived solution, called the left
conjugation solution. As mentioned earlier, in such a  trivial skew left brace the ideals are precisely the normal subgroups of $(G,\circ)$.
Of course in general one has several skew left brace structures on $G$. However, as a consequence on Hopf--Galois structures, Byott proves
in \cite{MR2011974} that if $(G,\circ) $ is a non-abelian simple group then there are only two possible skew left brace structures
on $G$, namely the trivial skew left brace $(G,\circ, \circ)$ and the almost trivial skew left brace $(G,\circ^{\op},\circ)$, where
$\circ^{\op}$ is the opposite operation of $\circ$. So, in this case we have $G^{(2)}=G$. The latter corresponds with the non-degenerate
solutions $(X,r)$ with all $\rho_x=\id$ and thus $\sigma_x=\lambda^{-1}_x$ for all $x\in X$. In this case the associated solution is
$r_G(a,b)=(a\circ b \circ a^{-1},a)$, called the right conjugation solution.

\begin{ex}\label{ex:byott}
    Byott in \cite{Byott-IP} exhibits the first class consisting  of an infinite family of simple skew left braces $(B,+,\circ)$
    which are not of abelian type and which do not arise from non-abelian simple groups. For any primes $p$ and $q$ such that $q$
    divides $p^p-1$ but does not divide $p-1$, there are (up to isomorphism) exactly two simple skew left braces of order $p^pq$,
    and with $(B,+)$ and $(B,\circ)$ solvable groups. The first three pairs of such $(p,q)$ and corresponding $n$ are: $(2,3)$ with
    $n=12$, $(3,13)$ with $n=351$, and $(5,11)$ with $n=34\,374$. The first instances of this family can be found in the \GAP\ package
    \emph{YangBaxter} \cite{YangBaxter}, i.e., \texttt{SmallSkewbrace(12,22)} and \texttt{SmallSkewbrace(12,23)}. Both of these skew
    left braces have order $12$ and $(B,+)\cong \Alt(4)\cong(C_2\times C_2)\rtimes C_3$ and $(B,\circ)\cong C_3 \rtimes C_4$.
    
    Byott mentions that there are only three types of groups of order $n=p^pq$. The Sylow $p$-subgroup $P$ or the Sylow $q$-subgroup
    $Q$ has to be normal (and thus invariant) and the three groups are respectively the direct product $P \times Q$ and the semi-direct
    products $P\rtimes Q$ and $Q \rtimes P$, of which he utilizes the latter two. A crucial property of those last two groups is that
    neither contains elements of order $pq$ and the latter does not contain a normal subgroup of order $p^i$ for some $i$. This is
    well-known and follows from the assumptions on $p$ and $q$ as $\GL_p(\mb{F}_p)$ will be the linear group of the smallest dimension
    containing an element of order $q$. To yield a skew left brace structure $B$ one needs that $P$ is normal and elementary abelian
    (hence one may identify $P$ with the additive group of the $\mb{F}_p$-vector  space $V=\mb{F}_p^p$), and the group has normal
    subgroups of order $1$, $p^p$ and $n$ and no subgroups of order $p^iq$ for $1\le i<p$. It is then shown that $(B,+)\cong P\rtimes Q$
    and $(B,\circ)\cong Q\rtimes P$.

    As an application of Theorem~\ref{simpleNL}, we show that these simple skew left braces yield simple non-degenerate solutions
    of size $p^p (q-1)$. To do so, we recall an explicit construction from \cite{Byott-IP}. Therefore, let $J\in\GL_p(\mb{F}_p)$ be
    a Jordan matrix with eigenvalue $1$. Then there exists a matrix $M \in\GL_{p}(\mb{F}_p)$ of order $q$ with $JMJ^{-1}=M^p$.
    Put $V=\mb{F}_p^p$ and consider the natural left action of the subgroup $\free{M}\s\Aut(V)$ on $V$ and the associated semi-direct product
    group $B=V\rtimes\free{M}=\{(v,M^i):v\in V\text{ and }0\le i<q\}$, written additively. Clearly $(B,+)\cong C_p^p\rtimes C_q$.
    Then, according to \cite{Byott-IP}, one can construct a regular subgroup $(B,\circ)$ of the holomorph $\Hol(B,+)=(B,+)\rtimes\Aut(B,+)$
    with generating elements: \[x=((0,M),\id),\quad y_v=((v,I),\gamma_{(-v,I)})\text{ for }v\in V_0,\quad z=((e_p,I),\gamma_{(-e_p,J)}),\]
    where $V_0=\mb{F}_pe_1\oplus\dotsb\oplus\mb{F}_pe_{p-1}$, $\{e_1,\dotsc,e_p\}$ is the standard basis of $V$, and $\gamma_a\in\Aut(B,+)$
    denotes the conjugation by $a\in B$, i.e., $\gamma_a(b)=a+b-a$ for $b\in B$. Then one can prove that the skew left brace $(B,+,\circ)$
    is simple. Denote by $Q$ the cyclic group of order $q$ generated by $x$. We claim that the set $X=\bigcup_{a\in B} \gamma_a(Q)\setminus\{1\}$
    generates $(B,+)$ and $B$ acts transitively on $X$. We first prove the latter. Note that $X$ coincides with the set of elements of order
    $q$ in $(B,+)$. Hence, $X$ is closed under the action of $B$, as all $\lambda_a$ and $\sigma_a$ maps are automorphisms of $(B,+)$.
    To show that the action is transitive, it is sufficient to prove that we can reach, under the action, any element of $Q$ starting from
    $x\in Q$, as $\sigma$ coincides with conjugation. For this, one first notices that
    \begin{align*}
        z+\lambda_z(x)-z &=z+(-e_p+JMJ^{-1}e_p,JMJ^{-1})-z\\
        &=(e_p-e_p+JMJ^{-1}e_p-JMJ^{-1}e_p,M^p)=(0,M^p)=px.
    \end{align*}
    As $p$ is invertible modulo $q$, one then can inductively reach all elements in $Q\setminus\{1\}$. Hence, this shows the transitivity on $X$.
    Rests to show that the set $X$ additively generates $B$. Hence, $X$ coincides with all elements of $B$ that are not of $p$ power order, i.e.,
    the complement of $V\rtimes \{I\}$. Thus, \[ X=\{((v,M^i),\lambda_{(v,M^i)}):v\in V\text{ and }0<i<q\}.\]
    As $(v,I)=(v, M)-(0,M)$, the claim follows and $X$ indeed generates $(B,+)$. Hence, by Theorem~\ref{CorSummary}, the solution restricted
    to $X$ is a simple solution of size $p^p(q-1)$. 
\end{ex}

We end with a general remark on solutions of prime order.

\begin{ex}\label{exsemipri}
    All indecomposable non-degenerate solutions of prime order are simple solutions. 
\end{ex}

This can easily be proven. First, we say that a morphism $f\colon(X,r)\to(Y,s)$ of solutions is a covering map if it is
surjective and all the fibres $f^{-1}(y)$ for $y\in Y$ have the same cardinality. A non-degenerate solution $(X,r)$ is said
to be transitive if $\mc{S}(X,r)$ acts transitively on $X$. By\cite[Lemma 1]{MR3958100} (for non-degenerate involutive solutions)
and by \cite[Lemma 3.2]{MR4388351} (for the general case), if $f\colon(X,r)\to(Y,s)$ is an epimorphism of solutions then $f$
is a covering map. In particular, if $|X|$ is prime then $(X,r)$ is a simple solution.

In \cite{MR1848966} Etingof, Guralnick, and Soloviev have described indecomposable finite non-degenerate solutions $(X,r)$ of the
Yang--Baxter equation that are of prime order. It turns out that then $X$ can be equipped with an abelian group structure $(X,+)$
(and thus we may assume that $X=\mb{Z}_p$) such that the solution $(X,r)$ is affine, i.e., it is of the form
\[r(x,y)=(\alpha(x)+\beta(y)+a,\gamma(x)+\delta(y)+b)\] for some $\alpha,\beta,\gamma,\delta\in\End(X,+)$ and $a,b\in X$.
They are precisely the solutions of the following type built on triples $a,b,c\in X$ with $a,b\ne 0$
such that $ab\ne 1$, and pairs $c_1,c_2\in X$ such that $(c_1,c_2)\ne(0,0)$:
\begin{enumerate}
    \item $(x,y)\mapsto (ay+(1-ab)x+c,bx-a^{-1}c)$,
    \item $(x,y)\mapsto (ay+c,bx+(1-ab)y-bc)$,
    \item $(x,y)\mapsto (y+c_1,x+c_2)$.
\end{enumerate}
The solutions of the third type are the solutions described in Proposition~\ref{SimpleLYU}
and the solutions of the first and second type are solutions determined by a simple skew left brace $B$.

It is interesting to note that all these solutions satisfy $\lambda_x=\lambda_y$ for all $x,y\in X$ or $\rho_x=\rho_y$ for all $x,y\in X$,
and thus are of the type investigated in \cite[Section 8]{CJKVA2023}, where it is shown that the Yang--Baxter algebras
corresponding to such solutions always are left, respectively right, Noetherian.

\section*{Acknowledgement}

The first named author is partially supported by the Fonds voor Wetenschappelijk Onderzoek (Flanders) -- Krediet voor wetenschappelijk
verblijf in Vlaanderen grant V512223N. The third author is supported by National Science Centre grant 2020/39/D/ST1/01852 (Poland).
The fourth author is supported by Fonds Wetenschappelijk Onderzoek Vlaanderen (grant 1229724N).

\bibliographystyle{amsplain}
\bibliography{refs}

\end{document}